\documentclass[article,12pt]{amsart}
\usepackage{amsmath,amsthm,amssymb,amscd}
\usepackage[all,cmtip,color]{xy}
\usepackage[german,english]{babel}
\usepackage{amsmath}
\usepackage{amssymb}
\usepackage{mathrsfs} 
\usepackage{bbm}
\usepackage{color}
\usepackage{dsfont}
\usepackage{todonotes}
\usepackage{url}

\usepackage[all,cmtip]{xy}
\usepackage[margin = 2.8cm,headsep=1cm]{geometry}
\usepackage{colonequals,enumerate}

\DeclareMathOperator{\Hom}{\mathsf{Hom}}

\DeclareMathOperator{\codim}{codim}

\def\llp{\mathopen{(\!(}}
\def\llb{\mathopen{[\![}}

\def\rrp{\mathopen{)\!)}}
\def\rrb{\mathopen{]\!]}}

\DeclareMathOperator{\Oc}{\mathcal{O}}

\DeclareMathOperator{\colim}{\mathsf{colim}}
\renewcommand{\lim}{\mathsf{lim}}

\DeclareMathOperator{\res}{res}

\DeclareMathOperator{\id}{id}

\DeclareMathOperator{\Set}{Set}

\DeclareMathOperator{\Lie}{Lie}

\DeclareMathOperator{\Gal}{Gal}
\DeclareMathOperator{\Aut}{\mathsf{Aut}}

\DeclareMathOperator{\GL}{GL}

\DeclareMathOperator{\SL}{SL}
\DeclareMathOperator{\M}{M}

\DeclareMathOperator{\Spec}{\mathsf{Spec}}

\DeclareMathOperator{\Sym}{Sym}

\newcommand{\BA}{{\mathbb{A}}}

\newcommand{\BC}{{\mathbb{C}}}

\newcommand{\BG}{{\mathbb{G}}}

\newcommand{\BL}{{\mathbb{L}}}
\newcommand{\BM}{{\mathbb{M}}}
\newcommand{\BN}{{\mathbb{N}}}

\newcommand{\BP}{{\mathbb{P}}}
\newcommand{\BQ}{{\mathbb{Q}}}
\newcommand{\BR}{{\mathbb{R}}}

\newcommand{\BX}{{\mathbb{X}}}

\newcommand{\BZ}{{\mathbb{Z}}}

\newcommand{\FA}{{\mathcal A}}
\newcommand{\FB}{{\mathcal B}}
\newcommand{\FC}{{\mathcal C}}
\newcommand{\FE}{{\mathcal E}}

\newcommand{\FI}{{\mathcal I}}

\newcommand{\FL}{{\mathcal L}}
\newcommand{\FM}{{\mathcal M}}
\newcommand{\FN}{{\mathcal N}}
\newcommand{\FO}{{\mathcal O}} 
\newcommand{\FP}{{\mathcal P}}
\newcommand{\FQ}{{\mathcal Q}}

\newcommand{\FS}{{\mathcal S}}

\newcommand{\FV}{{\mathcal V}}

\DeclareMathOperator{\Cent}{Cent}

\newcommand{\Var}{\mathrm{Var}}
\newcommand{\cM}{\mathcal{M}}

\newcommand{\Ko}{\mathrm{K}_0}

\renewcommand{\M}{\mathrm{M}}

\newcommand{\K}[1]{\mathrm{\bf K}(#1)}

\newcommand{\mcL}{\mathcal{L}}

\newcommand{\mcC}{\mathcal{C}}

\newcommand{\mcM}{\mathcal{M}}

\newcommand{\set}[1]{\left\{ #1 \right\}}

\newcommand{\abs}[1]{\left\lvert#1\right\rvert}
\newcommand{\eL}{\mathbb{L}}

\newcommand{\Gm}{{\mathbb{G}_\mathrm{m}}}
\newcommand{\un}{\mathds{1}}

\newcommand{\PSF}{\mathrm{PSF}}

\newcommand{\ACVF}{\mathrm{ACVF}}
\newcommand{\ACF}{\mathrm{ACF}}
\newcommand{\VF}{\mathrm{VF}}
\newcommand{\Vol}{\mathrm{Vol}}

\newcommand{\val}{\mathrm{val}}
\newcommand{\oST}{\overline{\mathrm{ST}}}

\newcommand{\RV}{\mathrm{RV}}
\newcommand{\RES}{\mathrm{RES}}
\newcommand{\kk}{\mathbf{k}}
\newcommand{\rv}{\mathrm{rv}}
\newcommand{\valrv}{\val_{\rv}}
\newcommand{\Gam}{\Gamma}
\newcommand{\Isp}{\mathrm{Isp}}
\newcommand{\eu}{\mathrm{eu}}
\newcommand{\an}{\mathrm{an}}
\newcommand{\Ext}{\mathrm{Ext}}

\newtheorem{theorem}[subsubsection]{Theorem}
\newtheorem{proposition}[subsubsection]{Proposition}

\newtheorem{corollary}[subsubsection]{Corollary}

\newtheorem{lemma}[subsubsection]{Lemma}
\newtheorem{assumption}[subsubsection]{Assumption}
\theoremstyle{definition}
\newtheorem{definition}[subsubsection]{Definition}

\newtheorem{rmk}[subsubsection]{Remark}

\newtheorem{example}[subsubsection]{Example}

\numberwithin{equation}{subsection}

\begin{document}

\author{Arthur Forey}
\address{Univ. Lille, CNRS, UMR 8524-Laboratoire Paul Painlevé, F-59000 Lille, France} 
  \email{arthur.forey@univ-lille.fr}
\urladdr{https://pro.univ-lille.fr/arthur-forey/}

\author{Fran\c {c}ois Loeser}
\address{Sorbonne Universit\'e, Institut de Math\'ematiques de Jussieu-Paris
Rive Gauche, CNRS. Campus Pierre et Marie Curie, case 247, 4 place Jussieu, 75252 Paris cedex 5, France.
}
\email{francois.loeser@imj-prg.fr}
\urladdr{https://webusers.imj-prg.fr/$\sim$francois.loeser/}

\author{Dimitri Wyss}
\address{EPFL/SB/ARG, Station 8, CH-1015 Lausanne, Switzerland
}
\email{dimitri.wyss@epfl.ch}
\urladdr{https://people.epfl.ch/dimitri.wyss}

\title{An orbifold formula for algebraic stacks}

\begin{abstract} We study motivic integration on varieties that are birational to  smooth algebraic stacks and prove a generalization of the orbifold formula for varieties with finite quotient singularities. This leads in particular to a new way of computing stringy $E$-functions and we also give applications to the study of klt-singularities. 

Along the way we study Cluckers-Loeser motivic integration on algebraic stacks, which might be of independent interest. The proof relies furthermore on ideas from Hrushovski-Kazhdan integration and Bruhat-Tits theory. 
\end{abstract}

\maketitle 

\tableofcontents

\section{Introduction}

\subsection{}

For $G$ a finite subgroup
of $\SL_2 (\BC)$, McKay established in 
\cite{mckay} an explicit correspondence between non-trivial irreducible representations of $G$ and exceptional divisors of the minimal resolution of the rational double point singularity
$\BC^2/G$. More generally, for $G$ a finite group
of automorphisms of a smooth complex algebraic variety $M$, the study of the connections between the geometry of resolutions of the quotient $M / G$ and the $G$-equivariant geometry of $M$ is what is nowadays known under the umbrella appellation of  ``McKay correspondence".
As an illustration of this paradigm, 
Batyrev proved in \cite{batyrev_jems}
 a formula expressing the stringy $E$-polynomial of the quotient in terms of the orbifold $E$-polynomial which is defined in terms of the group action.
Slightly later, a different approach of this result was given in \cite{DL2002} using what is now called the orbifold volume formula.
Both \cite{batyrev_jems} and  \cite{DL2002} use the theory of motivic integration  developed in  \cite{DL99}, which extends the theory
of motivic integration on smooth varieties, originally  introduced by Kontsevich in \cite{Ko95} in order to prove that two birational Calabi-Yau varieties have the same Hodge numbers, to singular varieties.
Yasuda later developed in \cite{Ya06} a theory of motivic integration for smooth Deligne-Mumford stacks which appears to provide  a very convenient framework for formulating the most general version of the orbifold volume formula. It can be stated as follows (in a formulation somewhat stronger than  the  ones in \cite{DL2002} and  \cite{Ya06}, where different Grothendieck rings were used).

Let $\FM$ be a smooth Deligne-Mumford stack over an algebraically closed field $k$ of characteristic $0$ and $\pi: \FM \to M$ its coarse moduli space. Assuming $\pi$ to be birational, one can define the orbifold measure $|\omega_{orb}|$ on the arc space $M(k \llb t\rrb)$ with values in a suitably localized Grothendieck ring of varieties $\Ko(\Var_k)_{loc}$ defined in  (\ref{kloc}). The total volume of $M(k \llb t\rrb)$ is computed by the orbifold volume formula as the weighted motivic class of the cyclotomic inertia stack $I_{\hat{\mu}}\mcM = \Hom(B\hat{\mu},\mcM)$ i.e.
\begin{equation}\label{off}\int_{M(k \llb t\rrb)} |\omega_{orb}| = [I_{\hat{\mu}}\mcM]^w = \sum_{a \in \BQ} \BL^{-a} [w^{-1}(a)],  \end{equation}
where $w: I_{\hat{\mu}}\mcM \to \BQ$ is a natural weight function (\ref{sec:weight}). If in addition there exists a crepant resolution $\tilde M \to M$, then one also has $\int_{M(k \llb t\rrb)} |\omega_{orb}| = \BL^{-\dim M}[\tilde M]$ giving rise to a motivic McKay correspondence. 
Formula \eqref{off} also led to a number of remarkable applications in arithmetic and algebraic geometry \cite{Ya17, GWZ20b,GWZ20,LW19, FLW} and furthermore, the right hand side of \eqref{off} admits a cohomological refinement, the Chen-Ruan orbifold cohomology \cite{CR04}.

\subsection{}The main result of this article is a generalization of \eqref{off} to smooth algebraic stacks $\FM$ admitting a birational  morphism $\pi:\FM \to M$ to a variety satisfying Assumption \ref{basas}, most of which are satisfied if $\pi$ is a good moduli space morphism \cite{Al13}. 

To state our theorem, let $U \subset M$ be a dense open such that $\pi^{-1}(U) \to U$ is an equivalence and consider 
\[M^\natural(k \llb t\rrb)= M(k \llb t\rrb) \cap U(k \llp t\rrp). \]
Then $M^\natural(k \llb t\rrb)$ admits a natural measure $|\omega_{M^\natural}|$, the analogue of the orbifold measure, whose volume can be computed by means of the finite order cyclotomic inertia stacks $I_{\mu_N}\mcM = \Hom(B\mu_N,\mcM)$ for $N\geq 1$. 

\begin{theorem}[\ref{mainlim}]\label{introthm}
The series $\sum_{N\geq 1} [I_{\mu_N} \mcM]^{w} \, T^N$ is rational of degree zero and we have
\begin{equation}\label{limform}
\int_{M^\natural(k \llb t\rrb)}|\omega_{M^\natural}|=-\lim_{T\to +\infty} \sum_{N\geq 1} [I_{\mu_N} \mcM]^{w} \, T^N
\end{equation}
in $\Ko(\Var_k)_{loc}$.
\end{theorem}
Theorem \ref{mainlim} below is quite a bit more flexible, in particular it allows to compute volumes of arcs based at any constructible $Z \subset M$ and more general measures vanishing along any simple normal crossing divisor on $\FM$. If $\FM$ is a DM-stack \eqref{limform} reduces to \eqref{off}, see Example \ref{dms}.

In \cite{GWZ24} a $p$-adic version of Theorem \ref{introthm} was proven and used to express certain BPS-invariants in Donaldson-Thomas theory as $p$-adic integrals. Here we give a slightly different, but related by \cite[Remark 6.7]{GWZ24}, application, showing in Section \ref{stef} how Theorem \ref{introthm} leads to a formula for Batyrev's stringy E-function \cite{batyrev_stringy} for $M = M_{r,d}$ the moduli space of semi-stable vector bundles of rank $r$ and degree $d$ on a smooth projective curve. In particular we recover the computation for $(r,d) = (2,0)$ of \cite{Kiem_Li} without the need of an explicit resolution of singularities. 

Furthermore, as for Deligne-Mumford stacks, the right hand side of \eqref{limform} admits a cohomological refinement in mixed Hodge structures, which will be the subject of upcoming work of the third author with S. Schlegel-Mejia.

As a third consequence, mostly of the proof of Theorem \ref{introthm}, we recover in Section \ref{tqt} a theorem of Braun and Moraga \cite{BM} about the behavior of klt-singularities under passage to $T$-quasi-torsors for $T$ a torus. 

Let us now mention the main ideas that go into the proof of Theorem \ref{introthm}, for which we work in the framework of Cluckers-Loeser motivic integration \cite{CL-2008}. For $N \geq 1$ consider the root stack $D^{1/N} = [\Spec(k\llb t^{1/N} \rrb)/\mu_r]$ and the infinite root stack $D^{1/\infty} = [\Spec(k\llb t^{1/\infty} \rrb)/\hat \mu]$. A twisted arc of order $r$ is a $k\llb t \rrb$-morphism $D^{1/N} \to \FM$ and under some mild assumptions on $\FM$ the set of twisted arcs $\FM(D^{1/N})$ admits the structure of an imaginary definable. 

Consider the morphism $\pi\colon \FM(D^{1/\infty})\to M(k\llb{t}\rrb)$. Its fibers are identified with the rational points of a convex region inside the Bruhat-Tits building of some large general linear group. We consider their volume for Hrushovski-Kazhdan integration \cite{HK}. Under our assumptions, we show that this convex region is in some sense contractible and the volume of the fibers is equal to one. By a formula of the first author and Yin \cite{FY}, we then express this volume as a limit of a generating series over twisted arcs of fixed order.  This process is reminiscent of the formation of the motivic Milnor fiber as a limit of the motivic zeta function. 

After a switch of limit and integral and an application of Fubini theorem, the formula is then reduced to the computation of the volume of the set of twisted arcs $\FM(D^{1/N})$. This goes by studying the fibers of $e_N\colon \FM(D^{1/N})\to I_{\mu_N}\mcM(k)$, where $e_N$ is given by restricting a twisted arc to $B\mu_N$. This is rather straightforward, going back to the analysis of Denef and Loeser in the Deligne-Mumford case \cite{DL2002}.

%

Let us also mention that Satriano and Usatine have also developed a theory of motivic integration on algebraic stacks, closer to Kontsevich's original approach using the geometry of jet schemes, with application to birational geometry \cite{SU24}.

The structure of the article is as follows: In Section \ref{pre} we recall the necessary background on motivic integration and algebraic stacks. In Section \ref{sec:voltwisted} we study twisted arcs in algebraic stacks from a definable point of view and we compute the volume of the fibers of $e_r$ in Theorem \ref{fixram}. In Section \ref{hkc} we study the fibers of the $\pi_r$ by means of Bruhat-Tits theory and in Section \ref{gof} we prove our main Theorem \ref{mainlim} and give some explicit examples. Finally Section \ref{apps} contains the aforementioned applications. 

\subsection*{Acknowledgements} 

Many ideas in this work were developed in parallel with the $p$-adic theory of \cite{GWZ24} and we warmly thank Michael Groechenig and Paul Ziegler for sharing their insights with us. We further thank Francesca Carocci, Young-Hoon Kiem,  Matthew Satriano, Jeremy Usatine  and Tanguy Vernet for interesting discussions on various aspects of the paper. 

F.L. was partially supported by the Institut Universitaire de France. D.W. was supported by the Swiss National Science Foundation [No. 218340].

\section{Preliminaries}\label{pre}

We fix a base field $k$ of characteristic zero containing all roots of unity. We work in the setting of Cluckers-Loeser motivic integration with algebraically closed residue field, that we recall below.

\subsection{Definable in Denef-Pas language}
Let $\ACF_k$ be the category of algebraically closed fields containing $k$. Let $\mcL_{DP,k}$ the three sorted language of Denef-Pas, with parameters $k\llp t\rrp$. We work in the $\mcL_{DP,k}$-theory of Henselian discretely valued fields of equicharacteristic zero containing $k\llp t\rrp$ with algebraically closed residue field. For the purposes of motivic integration, we shall restrict to the (non-elementary) class of models of the form $K\llp t\rrp$, where $K$ is an algebraically closed extension of $k$. 

\begin{definition}
A functor $X\colon\ACF_k\to \Set$ is a \emph{definable sub-assignment} (or simply definable set) if it is representable by an $\mcL_{DP,k}$-formula, meaning that there exists an $\mcL_{DP,k}$-formula $\varphi$ such that for every $K\in \ACF_k$, $X(K)=\varphi(K\llp t\rrp)$.

A functor $X\colon\ACF_k\to \Set$ is an \emph{imaginary sub-assignment} (or simply imaginary set) if it is representable by a quotient of a definable set by a definable equivalence relation, meaning that there exists  $\mcL_{DP,k}$-formulas $\varphi(x)$, $\psi(x_1,x_2)$, where $x_1,x_2$ and $x$ are variables with the same arity and sorts such that for every $K\in \ACF_k$, $\psi(K\llp t\rrp)$ is the graph of an equivalence relation $\sim_K$ on $\varphi(K\llp t\rrp)$ and $X(K)=\varphi(K\llp t\rrp)/\sim_K$.
\end{definition}

A map between two imaginary sets with graph representable by an imaginary set is called (abusively) a definable function. 

For example, given a affine algebraic variety $X$ over $k$, we associate three definable sets: $K\mapsto X(K)$, $K\mapsto X(K\llp t\rrp)$ and $K\mapsto X(K\llb t\rrb)$. We have a definable morphism $\res\colon X(K\llb t\rrb)\to X(K)$ induced by the residue map and an inclusion $X(K\llb t\rrb)\subset X(K\llp t\rrp)$. 

Throughout the paper, when describing definable sets and maps, we will often simply mention the $K$-points of the sets and maps, making sure that formulas representing them are uniforms across all $K$
running over algebraically closed extensions of $k$.

\subsection{Rings of constructible motivic functions} 
For $X$ definable, we consider the ring $\FC(X)$ of constructible motivic functions of Cluckers-Loeser. Recall that it is defined as $\FC(X)=\FP(X)\otimes_{\FP_0(X)}\FQ(X)$.  Here $\FP(X)$ is the ring generated over $\BA=\BZ\left[\eL,\eL^{-1},\left(\frac{1}{1-\eL^{-k}}\right)_{k\in \BN_{\geq 1} }\right]$ by functions of the form $\alpha\colon X\to \BZ$ and $\eL^{\alpha}$, where $\alpha \colon X\to \BZ$ is a definable function, $\FP_0(X)$ is the ring generated over $\BZ[\eL^{-1}]$ by characteristic functions of definable subsets of $X$. Finally, $\FQ(X)$ is the Grothendieck ring of definable sets over $X$ of the form $W\to X$ where $W\subset X\times K^r$ and the map is the projection. 

Let $N\geq 1$ be an integer. We define $\FP_{1/N}(X)$ as the ring generated over $\FP(X)$ by functions of the form $\eL^{\alpha/N}$, where $\alpha\colon X\to \BZ$ is definable, with the extra relation $(\eL^{1/N})^N=\eL$. Define $\FC_{1/N}(X)$ as $\FP_{1/N}(X)\otimes_{\FP_0(X)}\FQ(X)$. Elements of $\FC_{1/N}(X)$ can be integrated formally the same way as elements of $\FC(X)$.

For $N\geq 1$, we consider the degree $N$ ramified extension $k\llp t^{1/N}\rrp$ of $k\llp t\rrp$, with value group $\frac{1}{N}\BZ$. For $X$ a definable by a $\mcL_{DP,k}$-formula with parameters $k\llp t^{1/N}\rrp$, we now define a variant of the ring  of constructible motivic functions $\tilde{\mcC}_{1/N}(X)$, insisting that the valuation takes values in $\frac{1}{N}\BZ$. Define $\tilde{\FP}_{1/N}(X)$ as the ring generated over $\BZ\left[\eL,\eL^{-N},\left(\frac{1}{1-\eL^{-k/N}}\right)_{k\in \BN_{\geq 1} }\right]$ by functions of the form $N_\alpha$ and $\eL^{\alpha}$, where $\alpha \colon X\to \frac{1}{N}\BZ$ is a definable function and $N_\alpha(x)=\# \set{y\in \frac{1}{N}\BZ \mid 0<y\leq \alpha(x)}$. Then $\tilde{\FC}_{1/N}(X)=\tilde{\FP}_{1/N}(X)\otimes _{\FP_0(X)}\FQ(X)$. 

\subsection{Forms and measures} \label{fam}

By a $k\llb{t}\rrb$-variety we will mean a separated, integral and finite type scheme over $\Spec(k\llb{t}\rrb)$. Given such a $k\llb{t}\rrb$-variety $X$ of relative dimension $d$ and a $d$-form $\omega$ on the smooth locus of the generic fiber $X_{k\llp{t}\rrp}^{sm}$ we obtain a motivic measure $|\omega|$ on the definable set
\[K \longmapsto X^{sm}(K\llp{t}\rrp), \]
by writing definably-locally $\omega = f dx_1\wedge \dots \wedge dx_d$ and integrating $|f|$, see \cite{CL-2008}. Similarly, if $\omega \in H^0(X_{k\llp{t}\rrp}^{sm}, (\Omega_{X/k\llb{t}\rrb}^d)^{\otimes N})$ is a pluricanonical form we obtain a measure $|\omega|$ valued in $\tilde{\mcC}_{1/N}(X)$ by locally integrating $|f|^{1/N}$.
By extending trivially we obtain a measure on all of $K \mapsto X(K\llp{t}\rrp)$ and since $X/k\llb{t}\rrb$ is separated this also gives a measure on $K \mapsto X(K\llb{t}\rrb) \subset X(K\llp{t}\rrp)$, still denoted by $|\omega|$. 

The measure $|\omega|$ will of course depend on the choice of $\omega$, but in many situations one can construct measures in a more intrinsic manner. One such construction, see \cite[Section 4.1]{Ya17}, associates to any invertible $\Oc_X$-submodule $\FI \subset (\Omega^d_{X/k\llb{t}\rrb})^{\otimes N} \otimes \mathcal{K}(X)$, where $N$ is any positive integer and $ \mathcal{K}(X)$ is the sheaf of total quotient rings on $X$, a measure $|\omega_{\FI}|$ valued in $\tilde{\mcC}_{1/N}(X)$ as follows:\\
Pick a Zariski-cover $X=\bigcup_i U_i$ which trivializes $\FI$. A trivializing section of $\FI_{|U_i}$ gives rise to a pluricanonical form $\omega_i$ on $U_i^{sm}$ and in particular a measure $|\omega_i|$ on $K \mapsto U_i(K\llb{t}\rrb)$ valued in  $\tilde{\mcC}_{1/N}(X)$. The assumption on $\FI$ being invertible implies that the $|\omega_i|$ glue together to give a measure $|\omega_{\FI}|$ on $K \mapsto X(K\llb{t}\rrb)$. 

If $X/k\llb{t}\rrb$ is smooth, $\Omega^d_{X/k\llb{t}\rrb}$ is itself invertible and we write $|\omega_X|$ for $|\omega_{\Omega^d_{X/k\llb{t}\rrb}}|$, sometimes called the Weil or canonical measure. If furthermore $D= \sum_{i\in I} u_i D_i$ is a relative $\BQ$-divisor on $X$ we write $|\omega_{X,D}|$, or $|\omega_{D}|$ if the space $X$ in understood, for the measure associated with $(\Omega^d_{X/k\llb{t}\rrb}(D))^{\otimes N}$, where $N\geq 1$ is such that $ND$ is a $\BZ$-divisor. 

Now let $M = X/\Gamma$ be the quotient of a smooth variety $X$ by a finite group $\Gamma$ and $D= \sum_{i\in I} u_i D_i$ a relative $G$-invariant $\BQ$-divisor. Then there exists a unique relative $\BQ$-divisor $E$ on $M$ such that $q^*\Oc(K_M+E)^{\otimes N} \cong (\Omega^d_{X/k\llb{t}\rrb}(D))^{\otimes N}$ for a suitable $N \geq 1$ \cite[Lemma 7.2]{Ya17} . We write $|\omega_{D,orb}|$ for the measure associated with $\FI = \Oc(K_M+E)^{\otimes N}$. If $D=0$  we write simply $|\omega_{orb}|$, the orbifold measure on $M$. 

\subsection{Integration}

Let $X$ be a definable set, and $\abs{\omega}$ a definable volume form on $X$. Cluckers and Loeser define a subgroup of $\FC(X)$ consisting of integrable motivic constructible functions with respect to $\abs{\omega}$, and for such a $\varphi\in \FC(X)$, they associate
$\int_X \varphi \abs{\omega}\in \FC(\Spec{k})$, the integral of $\varphi$ with respect to $\abs{\omega}$. 

The integral is also defined over a general base. Let $f\colon X\to S$ be a definable map, and $\abs{\omega_{X/S}}$ a definable volume form on the fibers of $f$. Let $\varphi\in \FC(X)$ such that $\varphi$ is integrable with respect to $\abs{\omega_{X/S}}$ on the fibers of $f$. Then there is a constructible motivic function $\psi$ which for every $s\in S$ away from a subset of $S$ of smaller dimension satisfy
\[
\psi(s)=\int_{f^{-1}(s)} \varphi \abs{\omega_{X/S}}.
\]

This integration over a general base allows for a Fubini theorem, see \cite[15.2]{CL-2008}. 

\begin{theorem}
\label{thm:fubini}
Let  $f\colon X\to S$ a definable map, with equidimensional fibers. Let $\abs{\omega_X}$ and $\abs{\omega_S}$ be definable volume forms on $X$ and $S$. They induce a Leray residue form $\abs{\omega_{X}/f^*(\omega_S)}$ on the fibers of $f$. Let $\varphi\in \FC(X)$. Assume that $\varphi$ is integrable with respect to $\abs{\omega_X}$. Then $\varphi$ is integrable with respect to the quotient form $\abs{\omega_{X}/f^*(\omega_S)}$ on the fibers of $f$, let $\psi\in \FC(S)$ be the relative integral satisfying for $s\in S$ away from a subset of smaller dimension 
\[
\psi(s)=\int_{f^{-1}(s)} \varphi \abs{\omega_{X}/f^*(\omega_S)}.
\]
Then $\psi$ is integrable with respect to $\abs{\omega_S}$ and we have
\[
\int_X\varphi \abs{\omega_X}=\int_S\psi \abs{\omega_S}.
\]
\end{theorem}

\subsection{Switching limit and integral}
Let $X$ be a definable set and $\varphi\in \FC(X\times \BZ_{>0})$. By \cite[Theorem 14.4.1]{CL-2008}, the series $\sum_{n\geq 1} \varphi(\cdot,n) T^n$ is rational, in the sense that it is equal to $P(T)/Q(T)$, where  $P(T)$ is a polynomial in $\FC(X)[T]$ and  $Q(T)\in \BZ[\eL,\eL^{-1}][T]$ is a product of terms of the form $1-\eL^\alpha T^\beta$, with $\alpha\in \BZ$, $\beta\in \BZ_{>0}$. If $P$ and $Q$ have the same degree $d$, we say that the limit as $T\to+\infty$ of $\sum_{n\geq 1} \varphi(\cdot,n) T^n$ exists and define 
\[
\lim_{T\to+\infty} \sum_{n\geq 1} \varphi(\cdot,n) T^n= P_d/Q_d,
\]
where $P_d$ and $Q_d$ are the top degree coefficients of $P$ and $Q$. Note that the specific form of $Q$ ensures that $Q_d\in \pm \eL^{\BZ}$, hence the limit is in $\FC(X)$. 

The following proposition shows that this limit procedure can be switched with integration. 

\begin{proposition}
\label{prop:switch:lim:int}
Let $\varphi\in\FC(S\times X\times \BZ_{>0})$. Assume that for each $x\in X$, the limit $\lim_{T\to +\infty}  \sum_{n\geq 1} \varphi(\cdot,\cdot,n)T^n$ exists and as a function on $S\times X$ is integrable with respect to $S$. Assume also that $\varphi$ is integrable with respect to $S\times \BZ_{>0}$. Then the following limit exists and we have the equality
\[
\int_X \lim_{T\to +\infty}  \sum_{n\geq 1} \varphi(\cdot,\cdot,n)T^n=\lim_{T\to +\infty}  \sum_{n\geq 1}\int_X \varphi(\cdot,\cdot,n)T^n
\]
in $\FC(S)$.
\end{proposition}

\begin{proof}
Let $\varphi\in\FC(S\times X\times \BZ_{>0})$ as in the statement of the proposition. By \cite[Theorem 14.4.1]{CL-2008}, $\sum_{n\geq 1} \varphi(\cdot,x,n)T^n$ is rational, more precisely, there exists a finite set $I$,  and for $i\in I$, $a_i\in \BZ, b_i\in \BZ_{>0}$ and a polynomial $P=\sum_{j=1}^r{\alpha_j T^j}\in \FC(S\times X)[T]$ such that 
\begin{equation}
\label{eqn-sum-limit}
\sum_{n\geq 1} \varphi(\cdot,\cdot,n)T^n=\frac{\sum_{i=j}^r{\alpha_j T^i}}{\prod_{i\in I} (1-\eL^{a_i}T^{b_i})}.
\end{equation}

Since we assume that $\lim_{T\to +\infty}  \sum_{n\geq 1} \varphi(\cdot,\cdot,n)T^n$ exists, $r$ can be chosen such that $r=\prod_{i\in I} b_i$, in which case
\[
\lim_{T\to +\infty}  \sum_{n\geq 1} \varphi(\cdot,\cdot,n)T^n=\prod_{i\in I}(-\eL^{-a_i})\alpha_r.
\]
Its integral with respect to $S$ is then equal to $\prod_{i\in I}(-\eL^{-a_i}) \int_{x\in X} \alpha_r$. 

On the other hand, since the denominator of the right hand side of (\ref{eqn-sum-limit}) does not depend on $S\times X$, by linearity of the integral we have

\[\sum_{n\geq 1} \int_{x\in X}\varphi(\cdot,x,n)T^n=\frac{\sum_{i=j}^r{\int_{x\in X}\alpha_j(\cdot,x) T^i}}{\prod_{i\in I} (1-\eL^{a_i}T^{b_i})}.
\]
Hence $\lim_{T\to +\infty}\sum_{n\geq 1} \int_{x\in X}\varphi(\cdot,x,n)T^n$ exists and is equal to $\prod_{i\in I}(-\eL^{-a_i}) \int_{x\in X} \alpha_r$, which concludes the proof. 
\end{proof}

\subsection{Definable structure on ramified arcs} 
We consider subsets of ramified arcs as follows. Recall that by assumption, $k$ contains all roots of unity. For 
$K$ an extension of $k$ and $N\geq 1$, we consider the field extension $K\llp t^{1/N}\rrp$ of $K\llp t\rrp$ as $K\llp t^{1/N}\rrp\simeq K\llp t\rrp[X]/(X^N=t)\simeq K\llp t\rrp^N$, where the last equivalence is between $K\llp t\rrp$-vector spaces. So we view $K\llp t^{1/N}\rrp$ as the definable $K\llp t\rrp^N$, and observe that the multiplication is given by a definable linear map $K\llp t\rrp^{2N}\to K\llp t\rrp^{2N}$. The Galois group $\mu_N$ is described by $N$ definable linear endomorphisms of $K\llp t\rrp^N$. The valuation on $K\llp t\rrp$ extend uniquely to the field $K\llp t^{1/N}\rrp$, yielding a valuation with value group $\frac{1}{N}\BZ$ that we still denote $\val$. 

Given a definable $X\subset K\llp t\rrp^r$, the points of $X$ in $K\llp t^{1/N}\rrp$ is then described by a definable subset of $K\llp t\rrp^{rN}$. This subset depend on the formula describing $X$, not only on $X$, but we will ignore this issue notationally. We denote it by $X(K\llp t^{1/N}\rrp)$. If $X$ has dimension $d$, $X(K\llp t^{1/N}\rrp)$ has dimension $Nd$. 

A definable map $f\colon X\to Y$ induces a definable map $X(K\llp t^{1/N}\rrp)\to Y(K\llp t^{1/N}\rrp)$. The valuation map $\val\colon K\llp t^{1/N}\rrp^*\to \frac{1}{N}\BZ$ is described on $K\llp t\rrp^N\backslash\set{0}$ by a definable map as well: the map $N\val$ is definable. Since $K\llp t^{1/N}\rrp\simeq K\llp t\rrp[X]/(X^N=t)=\oplus_{0\leq k< N} X^k K\llp t\rrp$, if $x\in K\llp t^{1/N}\rrp^*$ is written $x=x_0+Xx_1+\dots+X^{N-1}x_N$, then $N\val(x)=\min\set{N\val(x_k)+k}$, which yields a definable function on $K\llp t\rrp^N\backslash\set{0}$.

By a definable subset of $K\llp t^{1/N}\rrp^r$, we mean a definable subset of $K\llp t\rrp^{rN}$ through the above identification. For $Y$ such a definable subset of $K\llp t^{1/N}\rrp^r$, we introduce the invariant definable set
\[
Y^{\mu_N}=\set{y\in Y\mid \sigma(y)=y \text{ for every }\sigma \in \mu_N},
\]
which is definable since the elements of the Galois group $\mu_N$ are definable functions. If $Y$ is of dimension $d$, then $Y^{\mu_N}$ is of dimension at most $d/N$. 

Starting from a definable $X\subset K\llp t\rrp^r$, we get that $X(K\llp t^{1/N}\rrp)^{\mu_N}$ is in definable bijection with $X$.

\begin{rmk}
Beware that our definition is more restrictive than that of a definable subset of $K\llp t^{1/N}\rrp^r$ in the language $\mcL_{DP,k}$ with parameters $k\llp t^{1/N}\rrp$. For example $t^{1/N}$ is only algebraic, we cannot distinguish it from its Galois conjugates using an $\mcL_{DP,k}$-formula with parameters $k\llp t \rrp$. 
\end{rmk}

\subsection{Integration on ramified arcs} \label{inrac}

Let $X$ be a definable subset of $K\llp t\rrp^r$. Since constructible motivic functions are built out of $\mcL_{DP,k}$-formulas, if $\varphi\in \mcC(X)$, we get naturally a constructible motivic function $\varphi_N$ in  $\widetilde{\mcC}_{1/N}(X(K\llp t^{1/N}\rrp))$ by interpreting those formulas in $K\llp t^{1/N}\rrp$. Using the identification $K\llp t^{1/N}\rrp^r\simeq K\llp t\rrp^{Nr}$, we view $\varphi_N$ as a constructible motivic functions in $\mcC_{1/N}(K\llp t\rrp)^{Nr})$. We denote the ring of such constructible function by $\mcC_{1/N}(X(K\llp t^{1/N}\rrp))$. 

Hence a constructible motivic function in  $\mcC_{1/N}(K\llp t^{1/N}\rrp)^r)$ can be integrated by viewing it as an element of $\mcC_{1/N}(K\llp t\rrp)^{Nr})$. Similarly, if $X\subset \mcC_{1/N}(K\llp t^{1/N}\rrp)^r)$, it is endowed with a canonical volume form $\abs{\omega_X}$ and we can define $\int_{X(K\llp t^{1/N}\rrp)} \varphi\abs{\omega_X}$ for $\varphi\in \mcC(X)$.

Note that with this definition, the volume of $tK\llb t^{1/N}\rrb$ is $\eL^{-N}$, whereas the volume of $t^{1/N}K\llb t^{1/N}\rrb$ is $\eL^{-1}$.

Since we normalized the value group as $\frac{1}{N}\BZ$ for ramified arcs, the restriction of $\varphi_N$ to $X(K\llp t^{1/N}\rrp)^{\mu_N}$ is identified to $\varphi$.

Let $X\subset K\llp t\rrp^r$ be definable of dimension $d$ and $\abs{\omega}$ a definable volume form on $X$. Up to some finite definable partition of $X$ it is induced by a constructible motivic function $f\in \mcC(X)$ and a choice of $d$-coordinates $1\leq i_1<\dots<i_d\leq r$ such that $\abs{\omega}=\abs{fdx_{i_1}\wedge\dots dx_{i_d}}$. We still denote by $\abs{\omega}$ the induced definable volume form on $X(K\llp t^{1/N}\rrp)$.

\subsection{The Hrushovski-Kazhdan volume}

The statement of our main theorem involves motivic integrals in the sense of Cluckers-Loeser, as recalled above. However, the proof requires to use also a variant defined by  Hrushovski and Kazhdan \cite{HK}. We outline here what we need of their construction. Note that the setup is different, since we need to work in the theory of algebraically closed valued field.

In this section, we  consider the first order theory $\ACVF$ of algebraically closed valued fields of equicharacteristic zero in the two-sorted language $\mcL_\ACVF$. The two sorts are $\VF$ and $\RV$. We put the ring language on $\VF$, with symbols $(0,1,+,-,\cdot)$, on $\RV$ we put the group language $(\cdot, ()^{-1})$, a unary predicate $\kk^\times$ for a subgroup, and operations $+ : \kk^2\to \kk$ where $\kk$ is the union of $\kk^\times$ and a symbol 0. We add a unary function $\rv : \VF^\times=\VF\backslash\set{0} \to \RV$. 

We will also consider the imaginary sort $\Gam$ defined by the exact sequence
\[
1\longrightarrow \kk^\times \longrightarrow \RV\longrightarrow \Gam\longrightarrow 0,
\]
together with maps $\valrv : \RV \to \Gam$ and $\val : \VF^\times \to \Gam$. 

If $L$ is a valued field, with valuation ring $\FO_L$ and maximal ideal $\FM_L$, define an $\FL$-structure by $\VF(L)=L$, $\RV(L)=L^\times/(1+\FM_L)$, $\kk(L)=\FO_L/\FM_L$, $\Gam(L)=L^\times/\FO_L^\times$. Note that the valuation ring is definable in this language because $\FO_L^\times=\rv^{-1}(\kk^\times(L))$. 

Set $K=k\llp{t}\rrp$. View $K$ as a fixed base structure, for the rest of this section, we will only consider $\mcL(K)$-structures, where $\mcL(K)$ is the language obtained by adjoining to $\mcL$ constants symbols for elements of $K$. Any valued field extending $K$ can be interpreted as an $\mcL(K)$-structure. Denote $\ACVF_K$ the $\mcL(K)$-theory of such algebraically closed valued fields. The theory $\ACVF_K$ admits quantifier elimination in the language $\mcL(K)$. Quantifier elimination was first proven by Robinson using a two sorted language, with one sort $\VF$ and one sort $\Gamma$ for the value group, see for example \cite{weispfenning_quantifier_1984}. 

We will use the notation $\VF^\bullet$ for $\VF^n$ for some $n$. 
The $\mcL(K)$-definable subsets of $\VF^\bullet$ are semi-algebraic sets, that is boolean combinations of sets of the form
\[
\set{x\in \VF^n\mid \val(f(x))\geq \val(g(x))},
\]
where $f$ and $g$ are polynomials with coefficients in $K$. Observe that constructible sets are semi-algebraic, since one can take $g=0$ in the definition. 

Denote by $\K{\VF_K}$ the free group of $\mcL(K)$-definable subsets of ${\VF}^\bullet$, with the following relations :
\begin{itemize}
\item $[X]=[Y]$ if there is a semi-algebraic bijection $X\to Y$
\item $[X]=[U]+[V]$ if $X$ is the disjoint union $X=U\cup V$.
\end{itemize}
Cartesian product endows $\K{\VF_K}$ with a ring structure. 

\begin{rmk}
Note that this framework allows us to consider general semi-algebraic subsets of $K$-varieties as studied for example by Florent Martin in \cite{martin_cohomology_2014}. We say that $S$ is a semi-algebraic subset $X$, for $X$ a $K$-scheme, if $S$ is a finite union $S=\cup S_i$ such that for every $i$, there is an open affine subset $U_i=\Spec{A_i}$ of $U$ such that $S_i\subseteq U_i$ is defined in $U_i$ by boolean combination of subsets of the form $\set{ y\in U_i^\an\mid \val(f(y))\leq r\cdot\val(g(y))}$, with $f,g\in A_i$, $r\in \mathbb{Q}$. Hence we can consider its class $[S]\in \K{\VF_K}$. 
\end{rmk}
 
\begin{rmk}
Hrushovski and Kazhdan use a slightly different definition for $\K{\VF_K}$. They define it as the group generated by isomorphism classes of definable sets $X\subseteq \VF^\bullet\times \RV^\bullet$, such that for some $n\in \BZ_{>0}$, there is some definable function $f : X\to \VF^n$ with finite fibers, with cut-and-paste relations (the function $f$ is not part of the data).  We can show that for such an $X$, there is some definable $X'\subseteq \VF^\bullet$, with a definable bijection $X\simeq X'$, see \cite[Lemma 8.1]{HK}. Hence both definitions lead to the same group. Nevertheless, this alternative presentation is very useful for defining motivic integration, since it amounts to relating $\K{\VF_K}$ and some group related to definable sets in the sort $\RV$, that we will now define. 
\end{rmk}

Define $\RV_K$ to be the category of objects $Y\subseteq \RV$, with definable functions as morphisms. The category $\RV_K[n]$ is the category of pairs $(Y,f)$, with $Y\subseteq \RV^\bullet$ definable and $f : Y\to \RV^n$ a definable finite-to-one function. A morphism between $(Y,f)$ and $(Y',f')$ is a definable function $g : Y\to Y'$. The category $\RES_K$ is the full subcategory of $\RV_K$ whose objects $Y$ satisfy $\valrv(Y)$ finite. One defines similarly $\RES_K[n]$ to be the full subcategory of $\RV_K[n]$ whose objects $(Y,f)$ satisfy $\valrv(Y)$ finite. 

Note that the definition of morphisms in $\RV_K[n]$ implies that $\RV_K[n]$ is equivalent to the category of definable sets $X\subseteq \RV^\bullet$ such that there exist a definable function $f : X \to \RV^n$ with finite fibers.

We also consider the full subcategory $\RES_K[n]$ of $\RV_K[n]$ which objects have finite image in $\Gamma$. 

For $\alpha=(\alpha_1,\dots \alpha_r)\in \Gamma^r$, set $V_{\alpha_i}=\valrv^{-1}(\alpha)\subset \RV$ and $V_\alpha=V_{\alpha_1}\times \dots V_{\alpha_r}$. Up to a finite definable partition, every object of $\RES_K[n]$ is included in some $V_\alpha$ for some  definable $\alpha$. Here $\alpha$ is in $\BQ^r$ since $\BQ$ is the definable closure of the value group of $K=k\llp t\rrp$. 

Up to choosing roots of $t$, an object $X$ in $\RES_K[n]$ can be sendt to an constructible set over $k$. The morphism depends on the choice of roots, but not the image, so it defines a class in the Grothendieck group of varieties over $k$, still denoted $[X]$. 

From those categories one form the graded categories 
\[
\RV_K[\leq n]:=\coprod_{0\leq i\leq n} \RV_K[i],
\]
\[
\RV_K[*]:=\coprod_{i\in \BZ_{\geq 0}} \RV_K[i],
\]
\[
\RES_K[\leq n]:=\coprod_{0\leq i\leq n} \RES_K[i],
\]
\[
\RES_K[*]:=\coprod_{i\in \BZ_{\geq 0}} \RES_K[i].
\]

Hrushovski and Kazhdan main result is an isomorphism 
\[\int \colon \K{\VF_K}\to \K{\RV_K[*]}/\Isp,
\]
where $\Isp$ is an explicitly described ideal of $\K{\RV_K[*]}$. 

They then proceed to define a morphism 
\[\FE_b:\K{\RV_K[*]}/\Isp\longrightarrow \K{\Var_k},\]
such that for $[X]_n\in \RES_K[n]$,  $\FE_b([X]_n)=[X]\in \K{\Var_k}$, and for $\Delta\subset \Gamma^n$ definable, $\FE_b(\val^{-1}(\Delta))=\eu_b(\Delta)[\Gm_{k}^n]$, where $\eu_b(\Delta)=\eu(\Delta\cap [-M,M)]^n)$ for some large enough $M$, and $\eu$ is the (o-minimal) Euler characteristic.

On then define a morphism
\[
\Vol\colon \K{\VF_K} \longrightarrow \K{\Var_k}
\]
by $\Vol=\FE_b\circ\int$.

\begin{proposition}
\label{prop:defvolx}
Let $f\colon X\to S$ a definable map in $\ACVF_K$. Then the map $x\in S(L\llp t\rrp)\mapsto \Vol(f^{-1}(x))$ is a constructible motivic function. 
\end{proposition}
\begin{proof}
By Hrushovski-Kazhdan \cite[Proposition 4.5]{HK}, the fibers of $X$ over $S$ are in definable bijection with lifts of definable subsets of $\RV$. More precisely, there exists for $0\leq i\leq d=\dim_S(Y)$ definable subsets $W_i\subset S\times \RV^\bullet$ with fibers over $S$ admitting finite-to-one projections to $\RV^i$. And there exists a finite definable partition of $X$ into pieces $X_i$ such that the fibers of $X_i$ over $S$ are in definable bijection $h_i\colon X_i \simeq \mathfrak{L}_i(W_i)$ over $S$. Here

\[ \mathfrak{L}_i(W_i)=\set{(x,s,y)\in \VF^i\times W_i\mid \pi_i(y)=\rv(x)}
\]
is the lift of $W_i$, where $\pi_i$ is any finite to one coordinate projection to $RV^i$ on the fibers of $W_i$. So we can replace $X$ by $\mathfrak{L}_i(W_i)$. 

By \cite[Corollary 3.25]{HK}, up to taking a finite definable partition of $W_n$ and working with one of the pieces, we can further assume that $W_i$ is of the form $Y\times_S\tilde\Delta$, where $Y\subset S\times V_{\alpha}$, with $\alpha=(\alpha_1,\dots,\alpha_r)\in \BQ^r$   and $\Delta\subset S\times \Gamma^s$ and $\tilde \Delta=\set{(s,x)\in S\times \RV^s\mid (s,\valrv(x))\in \Delta}$. 

We have $\Vol(f^{-1}(x))=\FE_b([Y_x])\FE_b(\val_rv^{-1}(\Delta_x)$, so it suffices to show that each of those is constructible motivic function. For $\FE_b([Y_x])$, by definition $\FE_b([Y_x]$ is the class of a definable subset of the residue field in bijection with $Y_x$, the definition is uniform in $x$ so we get some definable $Z\subset S\times L^n$ such that $\FE_b([Y_x]=[Z_x]$, showing that it is in $\FC(S)$. For $\FE_b(\val_rv^{-1}(\Delta_x)$, it is by definition equal to $\eu_b(\Delta)(\eL-1)^s$. By cell decomposition, we can take a finite definable partition of $S$ such that $\eu_b$ is constant on each piece, showing that $\FE_b(\val_rv^{-1}(\Delta_x)$ is in $\FC(S)$. 
\end{proof}

Let $f\colon X\to S$ a definable map in $\ACVF_K$ and $\omega\colon X\to \Gamma$ a definable function. 

Assume that $f$ and $\omega$ are described by some quantifier free formulas. Then for each algebraically closed field $L$ and integer $n$, we obtain a map
\[f_n\colon X(L\llp t^{1/n}\rrp) \longrightarrow S(L\llp t^{1/n}\rrp).\]
For a fixed $n$ it is definable in the Denef-Pas language, since it is described by a quantifier-free formula. The map $\omega$ also induces a map $\omega\colon X(L\llp t^{1/n}\rrp)\to \BQ$. We assume that $\omega$ has image in $\frac{1}{n}\BZ=\Gamma(L\llp t^{1/n}\rrp)$. Note that this condition is verified for example if $\omega$ is of the form $v\circ f$, where $f\colon X\to \VF$ is a definable function.

By the Cluckers-Loeser construction, there is a constructible motivic function $\tilde \varphi_n\in \mcC(S(L\llp t^{1/n}\rrp))$ such that for $x\in S(L\llp t^{1/n}\rrp)$, $\varphi_n(x)=\int_{{f_n}^{-1}(x)} \abs{n\omega_n}$. Note that $\tilde \varphi_n$ is in $\mcC(S(L\llp t^{1/n}\rrp))$ and not merely $\mcC_{1/n}(S(L\llp t^{1/n}\rrp))$ since the function $n\omega_n$ take integer values. Let $\varphi_n\in \mcC(S(L\llp t\rrp))$ be the restriction of $\tilde \varphi_n$ to $S(L\llp t\rrp))$.


\begin{definition}
Let $X$ be a definable set and $\omega$ be a definable function $X\to \Gamma$.  Assume that there is a definable bijection of class $C^1$  $h\colon X\to \mathfrak{L}_n(W)$ and a definable map $\omega'\colon W\to \Gamma$  where $\mathfrak{L}_n(W)=\set{(x,y)\in \VF^n\times W\mid \pi(y)=\rv(x)}$, $W$ a definable subset of $\RV^\bullet$ and $\pi$ a finite-to-one projection to $\RV^n$ and such that $\omega(x)=\omega'(h(x))-\val(\mathrm{Jac})(h)(x)$. 

 We say that $\omega$ induces the counting measure on $X$ if for every $w\in W$, we have $\omega'(w)=-\val(\pi(w)$. Note that this determines $\omega$ uniquely

We write $\mu^\#(X(K(\llp t\rrp)=\int_{X(K(\llp t\rrp)} \abs{\omega}$, since it does not depend on $\omega$. 
\end{definition}

The relevant example of such an object for this paper is  when $X$ is a subset of $\GL_d$ invariant under $\GL_d(\FO)$ and $\omega=\val(f)$, where $f$  is a regular function $f$ on $\GL_d$ such that the form $f \wedge_i dx_i$ is invariant under $\GL_d(F)$, with $x_i$ the standard affine coordinates of $\GL_d$. 

\begin{proposition}
\label{prop:defphinx}
With the previous notations, assume further that $\omega$ induces a counting measure on the fibers of $f$ and that the fibers of $f$ are bounded. Then there is a constructible motivic function $\varphi\in \mcC(S\times\BZ_{>0})$ such that for each $n>0$, $\varphi(\cdot,n)=\varphi_n$.
\end{proposition}

\begin{proof}
As in the previous proofs, by \cite[Proposition 4.5]{HK}, the fibers of $X$ over $S$ are in definable bijection with lifts of definable subsets of $\RV$. More precisely, there exists for $0\leq i\leq d=\dim_S(Y)$ definable subsets $W_i\subset S\times \RV^\bullet$ with fibers over $S$ admitting finite-to-one projections to $\RV^i$. And there exists a finite definable partition of $X$ into pieces $X_i$ such that the fibers of $X_i$ over $S$ are in definable bijection $h_i\colon X_i \simeq \mathfrak{L}_i(W_i)$ over $S$. Here

\[ \mathfrak{L}_i(W_i)=\set{(x,s,y)\in \VF^i\times W_i\mid \pi_i(y)=\rv(x)}
\]
is the lift of $W_i$, where $\pi_i$ is any finite to one coordinate projection to $RV^i$ on the fibers of $W_i$. We can also assume that $\omega$ factors through a function on $W_i$, still denoted by $\omega$. 

Since restricting to Henselian subfields commutes with $\ACVF$-definable bijection, we can replace $X$ by $\mathfrak{L}_i(W_i)$ for $i=0,\dots,n$. But $\dim_S(\mathfrak{L}_i(W_i))=i$, so  only $\mathfrak{L}_n(W_n)$ have non-zero volume, so the only piece we need to consider is $\mathfrak{L}_n(W_n)$. 

By \cite[Corollary 3.25]{HK}, up to taking a finite definable partition of $W_n$ and working with one of the pieces, we can further assume that $W_n$ is of the form $Y\times_S\tilde\Delta$, where $Y\subset S\times V_{\alpha}$, with $\alpha=(\alpha_1,\dots,\alpha_r)\in \BQ^r$   and $\Delta\subset S\times \Gamma^s$ and $\tilde \Delta=\set{(s,x)\in S\times \RV^s\mid (s,\valrv(x))\in \Delta}$. 

By orthogonality between residue field and value group, up to taking a finite definable partition of $Y$, the ``volume form" function $\omega$ on $W_n=Y\times_S\tilde\Delta$ can be assumed to only depend on $\Delta$. 

The integral can now be computed explicitly as a product of an integral of a lift of $Y$ and of a lift of $\Delta$.

For the first, observe that $Y(L\llp t^{1/n}\rrp)=\emptyset$ if $\alpha\notin \frac{1}{n}\BZ^r$. When $\alpha\in\frac{1}{n}\BZ^r$, the restriction of $Y(F\llp t^{1/n}\rrp)$ to $S(L\llp t)$ is independent of $n$ and is in definable bijection with some subset of $S\times L^\bullet$. This defines an element $\psi$ in the relative Grothendieck group of varieties $\FQ(S)$. Let $A$ be the subset of $n\in\BZ_{>0}$ such that $\alpha\in\frac{1}{n}\BZ^r$ which is definable in the Presburger language, and $\un_A$ its characteristic function, viewed as an element of $\FC(S\times\BZ_{>0})$. Then the integral we are considering is equal to $\psi \un_A \eL^{-s_0}$, so is an element of $\FC(S\times\BZ_{>0})$.
 
We now show that the integral over the lift of $\Delta$ is represented by an element of $\FC(S\times\BZ_{>0})$. For $x\in S(L\llp t\rrp)$ and $n\in \BZ_{>0}$, the integral we need to compute is equal to 
\[
\varphi_n(x)=(\eL-1)\sum_{\gamma\in \Delta_{x}\cap (\frac{1}{n}\BZ)^s}\eL^{-n\omega(x,\gamma)+nl(\gamma)}
\]
where $\Delta_{x}$ is the fiber of $\Delta$ over $x$ and $l(\gamma)=\gamma_1+\dots\gamma_s-s$.

By assumption, $\eL^{-n\omega(x,\gamma)+nl(\gamma)}$ depends only on $x$ and not $\gamma$, so all it remains to show is that the function mapping $(x,n)$ to the number of $\frac{1}{n}\BZ$-points of $\Delta_x$ is a constructible motivic function. Denote by $g(x,n)$ this function. Since the fibers of $f$ are bounded, $\Delta$ is bounded as well so $g(x,n)$ takes finite values. 

By cell decomposition, up to taking a finite partition of $\Delta$, we can assume that $\Delta_x$ is a convex rational polytope. Furthermore there exists some $N$ independent of $x$ such that $N\Delta$ is an integral convex polytope. 

By  \cite[Theorem 4.6.8]{stanley}, $\sum_{n\geq 1} g(x,n)T^n$ is a rational function with denominator $(1-T^N)^{s+1}$ and numerator a polynomial $g_x(T)$ of degree $N(s+1)$. It follows that the coefficients of $g_x(T)$ can be expressed as a linear combination of $g(x,i)$ with $1\leq i\leq N(s+1)$, each of which is then a constructible motivic function in $\FC(S)$. Expanding $\frac{1}{(1-T^N)^{s+1}}$ as a power series, we find that its coefficients are polynomial in $n$ hence determine a constructible motivic function in $\FC(\BZ_{>0})$. We have thus expressed $g(x,n)$ as a product of two constructible motivic functions, hence it is a constructible motivic function. \end{proof}

Recall from \cite[Definition 6.2]{FY} that a definable set $Z\subset \VF^n$ is said to be proper invariant if it is bounded and there is some $\alpha\in \Gamma$ such that $Z$ is a (possibly infinite) union of balls of radius $\alpha$. 
\begin{proposition}
\label{prop:vollimeq}
Let $f\colon X\to S$ a definable map in $\ACVF_K$ and $\omega\colon X\to \Gamma$ a definable function. 
Assume that the fibers of $f$ are proper invariant and that $\omega$ induces the counting measure on the fibers of $f$. 

Then for every $x\in S(L\llp t\rrp)$, the sum $\sum_{n\geq 1} \varphi_n(x)T^n$ is a rational function of degree zero, hence its limit as $T\to +\infty$ exists formally and we have
\[
\Vol(f^{-1}(x))=-\lim_{T\to +\infty}\sum_{n\geq 1} \varphi_n(x)T^n.
\]
Here $\varphi_n(x)=\mu^\#(f_n^{-1}(x))=\int_{f_n^{-1}(x)}\abs{\omega}$ as in Proposition~\ref{prop:defphinx}.
\end{proposition}

\begin{proof}
By Propositions~\ref{prop:defphinx} and~\ref{prop:defvolx}, the inputs of both sides are constructible motivic functions. For each $x$, the result follows from \cite[Theorem 8.11]{FY}.
\end{proof}

\subsection{A few notions around algebraic stacks} Throughout this text we consider an algebraic stack $\mcM$ of finite presentation over a base scheme $B$, where typically we will be interested in $B=\Spec(k\llb{t}\rrb)$ with $k$ a field of characteristic $0$. We always assume $\mcM$ to have affine stabilizers and separated diagonal. For any $T/B$ we write $\mcM(T)$ for both the groupoid of $T$-points and its set of isomorphism classes. 

\subsubsection{cd-quotient stacks}
The following condition will essentially allow us to reduce questions about motivic integration on $\mcM$ to quotient stacks.

\begin{definition}\label{def:cdstack}\cite{cdstack} A algebraic stack $\mcM$ is a cd-quotient stack if it admits a Nisnevich covering $\tau:[X/\GL_d] \to \mcM$. By this we mean that $X/B$ is a scheme, $\tau$ a representable morphism and the pullback along any $T \to \mcM$ is a Nisnevich covering.
\end{definition}

In \cite{cdstack} a list of examples of cd-quotient stacks is given. In particular any $\mcM$ admitting a good moduli space is a cd-quotient stack by \cite[Theorem 6.1]{AHR19}.

\subsubsection{$\FS$-completeness}
The notion of $\FS$-completeness \cite[Section 3.5]{AHH} of a morphism $f:\mcM \to \FN$ will allow us to compute the Hrushovski-Kazhdan volume of the fibers of $f$, see Section \ref{hkfib}.

For any DVR $R$ with uniformizer $t$, let
\[ 
\oST_{R}=[\Spec( R[S,T]/(ST-t)) /\BG_m],
\]
where $\BG_m$ acts on variables $S$ and $T$ with weights $1$ and $-1$ respectively. 

\begin{definition} A morphism $f:\mcM \to \FN$ between algebraic stacks is $\FS$-complete with respect to $R$ if for any morphism $\oST_{R} \to \FN$ which lifts to $\oST_{R}\setminus 0 \to \mcM$ there exists a unique extension $\oST_{R} \to \mcM$. 
\end{definition}

By \cite[Proposition 3.44]{AHH} any good moduli space map $\mcM \to M$, with $M$ separated, is $\FS$-complete with respect to any $R$. 

\subsubsection{Motivic classes of stacks}\label{kloc} Assume $S=\Spec(K)$ for $K$ a field of characteristic $0$. To any finite type $K$-stack $\mcM$ with affine stabilizers one can associate a class $[\mcM]$ in the localized Grothendieck group of varieties 
\[ \Ko(\Var_K)_{loc}=\Ko(\Var_K)[\BL^{-1},(\BL^n-1)^{-1}; \ n\geq 1].\]
 The ring $\Ko(\Var_K)_{loc}$ has been identified with the Grothendieck ring of $K$-stacks in \cite{Ek09} and $[\mcM]$ is simply the class of $\mcM$. To actually compute the class, one first stratifies $[\mcM]$ by quotient stacks \cite[Proposition 3.5.9]{Kr99} and then writes each stratum $\mcM_i$ as a quotient $[Y_i/\GL_{n_i}]$. We then have
\[ [\mcM] = \sum_i [Y_i][\GL_{n_i}]^{-1}.  \]

\begin{rmk}\label{munt} With this definition we have that the class of the classifying space $B\mu_N$ equals $1$ in $\Ko(\Var_K)_{loc}$, as we have an equivalence $B\mu_N \cong [\BG_m/\BG_m]$, with $\BG_m$ acting on itself with weight $N$.
\end{rmk}

\subsubsection{Cyclotomic inertia and weights}\label{sec:weight}

Let $\hat\mu$ be the inverse limit of the group of roots of unity $\{\mu_N\}_{N \geq 1}$. Given an algebraic stack $\mcM$ we define its twisted or cyclotomic inertia stack $I_{\hat{\mu}}\mcM$  as
\[I_{\hat{\mu}}\mcM = \Hom(B\hat{\mu},\mcM) = \colim_N \Hom(B\mu_N\mcM).  \]

If $\mcM$ is smooth over a field $K$ we can define a natural weight function  
\[  w:I_{\hat{\mu}}\mcM(K)\longrightarrow \BQ  \]
as follows:

First, consider a representation $\rho$ of $\hat{\mu}$ on a finite dimensional $K$ vector space $V$. Then $\rho$ splits as a direct sum of characters $\chi_1,\dots,\chi_{\dim V} \in \mathrm{Irr}_{\hat{\mu}} = \BQ/\BZ$. For $1\leq i\leq \dim V$ we write $c_i$ for the unique representative of $\chi_i$ satisfying $0<c_i\leq 1$ and define 
\[w(\rho) = \sum_i c_i.\]

 Now a point $y\in I_{\hat{\mu}}\mcM(K)$ consists of a pair $(x,\phi)$ with $x\in \mcM(K)$ and $\phi:\hat{\mu} \to \Aut(x)$. By smoothness of $\mcM$ and functoriality, the tangent complex $T_x\mcM$ of $\mcM$ at $x$ is quasi-isomorphic to a two term complex $\rho_{-1} \to \rho_0$ of $\hat{\mu}$-representations and we define 
\[ w(y) = w(\rho_0) - w(\rho_{-1}). \]
Notice that if $\mcM \cong [X/G]$ is a global quotient we have
\[ T_x\mcM \cong [ \Lie(G) \to T_{\tilde{x}}X],\]
where $\tilde{x}\in X$ is any lift of $x$ and the morphism is the derivative of the action map.

\section{Volumes of twisted arcs}
\label{sec:voltwisted}

\subsection{Equivariant arcs}\label{eqsec} Let $k$ be a field of characteristic $0$ containing all roots of unity. Let $G$ be a linear algebraic group over $k\llb{t}\rrb$ acting on the right on a separated finite type $k\llb{t}\rrb$-scheme $X$. Fix an integer $N \geq 1$. In what follows we write $\mu_N$ for both the Galois-group of $k\llp{t^{1/N}}\rrp /k\llp{t}\rrp$ and for the constant group scheme over $k\llb{t}\rrb$.

\begin{definition} 
For $N\geq 1$ and $\phi: \mu_N \to G$ a $k\llb{t}\rrb$-homomorphism the $\phi$-equivariant arcs of $X$ are
\[  X(k\llb{t^{1/N}}\rrb)^\phi = \{x \in X(k\llb{t^{1/N}}\rrb) \ |\ \forall \xi \in \mu_N \colon  x^\xi = x\phi(\xi)  \}. \]
For any constructible subset $Z \subset X_k$ we write $X(k\llb{t^{1/N}}\rrb)^\phi_Z$ for the space of arcs whose closed point lies in $Z$. 
\end{definition}
Notice that  for any $\xi \in \mu_N$ and any $x\in X(k\llb{t^{1/N}}\rrb)$ we have $(x^\xi)_{|k} = x_{|k}$ and hence $X(k\llb{t^{1/N}}\rrb)^\phi_Z=\emptyset$ unless $Z \cap  X^\phi \neq \emptyset$.

Since the action of $G$ on $X$ is algebraic, the condition of being $\phi$-equivariant is definable and thus the assignment
\[ X_N^\phi \colon K/k \longmapsto  X(K\llb{t^{1/N}}\rrb)^\phi, \] 
is definable as well. 

Now assume that $X/k\llb{t}\rrb$ is smooth and irreducible. We fix a relative $G$-invariant $\BQ$-divisor $D= \sum_{i\in I} u_i D_i$ with $u_i < 1$ and simple normal crossing support. This data defines a measure $|\omega_{D,orb}|$ on the quotient $X/\mu_N$, where $\mu_N$ acts on $X$ through $\phi$. Now for every $K/k$ there is a finite map $X(K\llb{t^{1/N}}\rrb)^\phi \to X/\mu_N(K\llb{t}\rrb)$ and we define the measure $|\omega_D|$ on $X^\phi_N$ as the pullback of $|\omega_{D,orb}|$.


\begin{proposition}\label{equivol} For any $\phi: \mu_N \to G$ and any $x \in X^\phi(k)$ we have
\[\int_{X^\phi_{N,x}} |\omega_D|= \BL^{-w(T_xX)}\prod_{i\in I_x} \frac{(1-\BL^{-1})\BL^{u_iw(\FN_{X/D_i,x})}}{1-\BL^{-1+u_i}}, \]
where $I_x = \{i \in I \ |\ x \in D_i\}$ and $\FN_{X/D_i,x}$ denotes the fiber at $x$ of the normal bundle of $D_i$ in $X$. We consider both $T_xX$ and $\FN_{X/D_i,x}$ as $\hat\mu$-representations via $\phi$. 
\end{proposition}
\begin{proof} Since $X$ is smooth and $x \in X^\phi(k)$ we can linearize the action of $\mu_N$ on $X$ around $x$ such that in local coordinates $x_1,\dots,x_n$ the action is diagonal and $I_x$ is a subset of $\{1,\dots,n\}$ i.e.  $D_i$ is given by the vanishing of $x_i$ for all $i\in I_x$. If we choose representatives of the weights $w_1,\dots,w_n \in \frac{1}{N}\BZ \cap (0,1]$, we get a definable bijection
\begin{align*}\rho\colon \BA^n(k\llb{t}\rrb) &\rightarrow X(k\llb{t^{1/N}}\rrb)^\phi_x \\
(y_1,\dots,y_n) &\mapsto (t^{w_1}y_1,\dots,t^{w_n}y_n). \end{align*}
If we write $q\colon X \to X/\mu_N$, then we have by definition and the change of variables formula

\[ \rho^*|\omega_D| = (\rho\circ q)^*|\omega_{D,orb}| = \prod_{i \in I_x}|t^{-u_i(w_i+\val(y_i))+w_i}dy_1\wedge \dots \wedge dy_n|. \]

Thus
\begin{align*}\int_{X(k\llb{t^{1/N}}\rrb)^\phi_x}|\omega_D|&=\BL^{-w(\phi,x)} \int_{\BA^n(k\llb{t}\rrb)} \prod_{i\in I_x}\BL^{u_i(w_i+\val(y_i))}|dy_1\wedge \dots \wedge dy_n|\\
& = \BL^{-w(\phi,x)}\prod_{i\in I_x} \frac{(1-\BL^{-1})\BL^{u_iw_i}}{1-\BL^{-1+u_i}},\end{align*}
which ends the proof.
\end{proof}

\subsection{Twisted arcs}

Let $\mcM$ be an algebraic stack over $k\llb{t}\rrb$. 

\begin{definition} A twisted arc of $\mcM$ over $k\llb{t}\rrb$ is a $k\llb{t}\rrb$-morphism 
\[D^{1/N}=  D_k^{1/N} = [\Spec(k\llb{t^{1/N}}\rrb)/\mu_N] \longrightarrow \mcM.  \]
\end{definition}
Notice that the restriction to the closed point of $D^{1/N}$ induces a map 
\[e_N:  \mcM(D^{1/N}) \longrightarrow I_{\mu_N}\mcM(k) = \Hom_k(B\mu_N,\mcM).\]
If $\mcM = [X/G]$ is a quotient stack we may further compose $e_N$ with $\mcM \to BG$ to obtain a map 
\begin{equation}\label{phitype}\overline{e}_N: \mcM(D^{1/N})  \to I_{\mu_N}BG \cong [\Hom_k(\mu_{N},G)/G].  \end{equation}
In order to describe the fibers of $\overline{e}_N$ we need the following

\begin{lemma}\label{galco} Assume $G$ is a reductive group scheme over $k\llb{t}\rrb$. 
\begin{enumerate}

\item  The inclusion $G(k\llb{t}\rrb) \to G(k\llb{t^{1/N}}\rrb)$ induces a bijection of non-abelian group cohomology
\[H^1(\mu_N,G(k\llb{t^{1/N}}\rrb))^0 \cong H^1(\mu_N, G(k\llb{t}\rrb)) \cong \Hom_k (\mu_r,G)/G,\]
where $H^1(\mu_N,G(k\llb{t^{1/N}}\rrb))^0$ denotes the inverse image of the trivial cocycle under $H^1(\mu_N,G(k\llb{t^{1/N}}\rrb)) \to H^1(\mu_N,G(k\llp{t^{1/N}}\rrp))$.

\item Assume further $H^1(k,G_k) = H^1(k\llp{t}\rrp,G) =\{1\}$, let $\mcM = [X/G]$ be a quotient stack and $\phi:\mu_N \to G$. Then  we have
  \[ \overline{e}_N^{-1}(\phi) = X(k\llb{t^{1/N}}\rrb)^\phi / G(k\llb{t^{1/N}}\rrb)^\phi. \]

\end{enumerate}
\end{lemma}
\begin{proof}(1)
Associated to the inclusion $G(k\llb{t^{1/N}}\rrb) \to G(k\llp{t^{1/N}}\rrp)$ we have an exact sequence of pointed sets \cite[Proposition I.36]{Se94}
\[ 1 \to G(k\llb{t}\rrb) \to G(k\llp{t}\rrp) \to \left(G(k\llp{t^{1/N}}\rrp)/G(k\llb{t^{1/N}}\rrb)\right)^{\mu_N} \to H^1(\mu_N,{G(k\llb{t^{1/N}}\rrb)}^0) \to 1. \]

By ramified descent, the set $\left(G(k\llp{t^{1/N}}\rrp)/G(k\llb{t^{1/N}}\rrb)\right)^{\mu_N}$ embeds into the extended building of $G$ over $k\llp{t}\rrp$. Taking the quotient by $G(k\llp{t}\rrp)$ identifies this subset with the Weil group orbits of $X_*(T) \otimes_{\BZ} \BZ/N\BZ$, for some fixed maximal torus $T$. The latter set is in bijection with $\Hom_k(\mu_N,G)/G$, since any homomorphism $\mu_N \to G$ is $k$-conjugate to a homomorphism $\mu_N \to T$. \\

(2) Clearly we get for every $[x] \in X(k\llb{t^{1/N}}\rrb)^\phi / G(k\llb{t^{1/N}}\rrb)^\phi$ an element in $\overline{e}_N^{-1}(\phi)$. Conversely, consider $y \in \overline{e}_N^{-1}(\phi)$ and let $\tilde{y}:\Spec(k\llb{t^{1/N}}\rrb) \to [X/G]$ be the composition with $\Spec(k\llb{t^{1/N}}\rrb) \to D^{1/N}$. Then the pull back of $X \to [X/G]$ along $\tilde{y}$ is a trivial $G$-torsor over $\Spec(k\llb{t^{1/N}}\rrb)$, since by assumption there are no non-trivial $G$-torsors over $\Spec(k)$ and thus over $\Spec(k\llb{t^{1/N}}\rrb)$. Acting by $\mu_N$ on a chosen section $\sigma$ gives rise to a $1$-cocycle $\phi_\sigma:\mu_N \to G(\llb{t^{1/N}}\rrb)$. By (1) we may choose $\sigma$ in such a way that $\phi_\sigma =\phi$ producing the desired element in $X(k\llb{t^{1/N}}\rrb)^\phi / G(k\llb{t^{1/N}}\rrb)^\phi$.
\end{proof}

For a cd-quotient stack $\mcM$ we equip the assignment
\[ K/k \longmapsto   \mcM(D_K^{1/N}), \]
with the structure of an imaginary as follows:

First assume that $\mcM \cong [X/\GL_d]$ is a quotient stack. Let $T_d \subset \GL_d$ denote a maximal torus, $W_d = N(T_d)/T_d$ the Weil group and $\BX_*(T_d) = \Hom(\BG_m,T_d)$ the cocharacter lattice. Since any morphism $\mu_N \to \GL_d$ factors through a maximal torus and they are all conjugate to each other, we have for any $K/k$ a bijection
\[ \Hom_K(\mu_N,\GL_d)/\GL_d  \cong (\BX_*(T_d) \otimes \BZ/N\BZ)^{W_d}.\]

Thus if we fix a representative $\phi:\mu_N \to \GL_d$ for any element in the finite set $\Hom_k(\mu_N,\GL_d)/\GL_d$ we get a decomposition

\begin{equation}\label{phide} [X/\GL_d](D_K^{1/N}) = \bigsqcup_{[\phi] \in \Hom_k(\mu_N,\GL_d)/\GL_d} \overline{e}^{-1}_{N}(\phi), \end{equation}
with $\overline{e}_N$ as in \eqref{phitype}. By Lemma \ref{galco} we have a bijection 
\[\overline{e}^{-1}_{N}(\phi)\cong X(K\llb{t^{1/N}}\rrb)^\phi / \GL_d(K\llb{t^{1/N}}\rrb)^\phi\] giving the desired structure of an imaginary. 

Now for a general $\mcM$ let $\tau: [X/\GL_d] \to \mcM$ be a Nisnevich covering, $K/k$ be an extension, and $x\in \Hom(D_K^{1/N}, \mcM)(K\llb{t}\rrb)$. Then $x$ gives rise to a $\mu_N$-invariant morphism $\tilde x: \Spec(K\llb{t^{1/N}}\rrb) \to \mcM$. By the Nisnevich property, the pullback of $\tau$ along $\tilde x$ is isomorphic to $\sqcup \Spec(K\llb{t^{1/N}}\rrb$ and the choice of any component will provide a lift of $x$ to $[X/\GL_d]$. We thus have a surjection
\begin{equation}\label{glsr} [X/\GL_d](D_K^{1/N})\to \mcM(D_K^{1/N}). \end{equation}

Furthermore by the Nisnevich property there exists a stratification $\mcM = \bigsqcup \mcM_i$ with locally closed substacks and sections $\sigma_i:\mcM_i \to [X/\GL_d]$ of $\tau$. Since $\tau$ is in particular étale, \eqref{glsr} gives for any extension $K/k$ a bijection
\begin{equation}\label{covem} \mcM(D_K^{1/N})_{\mcM_{i,k}} \cong  [X/\GL_d](D_K^{1/N})_{\sigma_i(\mcM_{i,k})},\end{equation}
which we use to give $\mcM(D_K^{1/N})$ the structure of an imaginary.

\begin{proposition}\label{preind} The imaginary structure on $\mcM(D^{1/N})$ is independent of the covering $\tau: [X/\GL_d] \to \mcM$ and the sections $ \sigma_i:\mcM_i \to [X/\GL_d]$.
\end{proposition}
\begin{proof} As in the construction of the imaginary structure we assume first that $\mcM \cong [X/\GL_d]$. If $\mcM \cong [X'/\GL_{d'}]$ is any other presentation of $\mcM$ we can consider a third presentation $\mcM \cong [X''/G'']$ with $X'' = X \times_{\mcM} X'$ and $G'' = \GL_d\times \GL_{d'}$. 

Now let $\phi \in \Hom_k(\mu_N,\GL_d)$ such that $X(K\llb{t^{1/N}}\rrb)^\phi$ is non-empty. Then since $X'' \to X$ is a $\GL_d'$-torsor, the fiber over any $x \in X(K\llb{t^{1/N}}\rrb)^\phi$ admits a section and by Lemma \ref{galco} gives an element in $X''(K\llb{t^{1/N}}\rrb)^{\phi''}$ for a unique $[\phi''] \in \Hom_k(\mu_N,G'')/G''$.
This way the projection $X'' \to X$ induces a  definable bijection 
\[X''(K\llb{t^{1/N}}\rrb)^{\phi''} / G''(K\llb{t^{1/N}}\rrb)^{\phi''} {\overset{\sim}{\longrightarrow}} X(K\llb{t^{1/N}}\rrb)^\phi / \GL_d(K\llb{t^{1/N}}\rrb)^\phi.\]

Applying the same reasoning to the projection $X'' \to X'$ gives the result for a quotient stack.

For the general case it remains to notice that given two stratifications $\mcM = \bigsqcup_{i\in I} \mcM_i = \bigsqcup_{j\in J} \mcM'_j$ and section $\sigma_i,\sigma_j'$ of $\tau: [X/\GL_d] \to \mcM$, the composition $\sigma'_j \circ \tau$ induces for all $i \in I$ and $j\in J$ a definable bijection 
\[  [X/\GL_d](D_K^{1/N})_{\sigma_i(\mcM_{i,k}\cap \mcM_{j,k})} {\overset{\sim}{\longrightarrow}} [X/\GL_d](D_K^{1/N})_{\sigma'_j(\mcM_{i,k}\cap \mcM_{j,k})}. \qedhere\]
\end{proof}

\subsection{Volumes of twisted arcs}\label{sec:vol-twist-arcs} In this section we prove an extension of Proposition \ref{equivol} to a smooth cd-quotient stack $\mcM$ over $k\llb{t}\rrb$. Let $D= \sum_{i\in I} u_i D_i$ a relative $\BQ$-Cartier divisor on $\mcM$ with $u_i< 1$ and simple normal crossing support \cite[Section 2]{Be17}. We define a motivic measure $|\omega_{\mcM,D}|=|\omega_D|$ on $\mcM(D^{1/N})$ as follows:

Let $K/k$ be an extension. If $\mcM \cong [X/\GL_d]$ is a global quotient stack we have the decomposition 
\[ [X/\GL_d](D_K^{1/N}) = \bigsqcup_{[\phi] \in \Hom(\mu_N,\GL_d)/\GL_d} \overline{e}^{-1}_{N}(\phi), \]
from \eqref{phide}, thus it is enough to define the measure on each 
\[\overline{e}^{-1}_{N}(\phi) \cong X(K\llb{t^{1/N}}\rrb)^\phi / \GL_d(K\llb{t^{1/N}}\rrb)^\phi.\]
Let 
\[q: X(K\llb{t^{1/N}}\rrb)^\phi  \longrightarrow \overline{e}^{-1}_{N}(\phi) \]
be the quotient. We write $\tilde{D}$ for the pullback of $D$ to $X$ and define for every definable subset $A \subset  \overline{e}^{-1}_{N}(\phi)$ 
\[ \int_A |\omega_D| := \frac{\BL^{w(\mathfrak{gl}_d)}}{[\Cent(\phi)]}\int_{q^{-1}(A)} |\omega_{\tilde D}|.  \]
This is justified by the fact that 
\[ \int_{\GL_d(K\llb{t^{1/N}}\rrb)^\phi} |\omega_{\GL_d}| = \BL^{-w(\mathfrak{gl}_d)}[\Cent(\phi)],\]
by Proposition \ref{equivol}. Notice also that $\Cent(\phi)$ is isomorphic to a product of general linear groups and thus its class is invertible in $\Ko(\Var_K)_{loc}$.

For a general cd-quotient stack $\mcM$ we have seen in \eqref{covem} that $\mcM(D_K^{1/N})$ embeds into some $[X/\GL_d](D_K^{1/N})$ and we thus can define the measure in this case just by restriction. One checks as in Proposition \ref{preind} that the measure constructed this way is independent of any choices.

Finally let 
\[e_N:  \mcM(D^{1/N}) \longrightarrow I_{\mu_N}\mcM(K) = \Hom(B\mu_N,\mcM)\]
denote the map induced by restriction to the closed point of $D^{1/N}$.

\begin{theorem}\label{fixram} 
 Let $(x,\phi) \in I_{\mu_N}\mcM(K)$ and $I_x = \{i \in I \ |\ x \in D_i\}$. Then we have 
\[ \int_{e_N^{-1}(x,\phi)}|\omega_D| = \BL^{-w(x,\phi)}[B\Aut(x,\phi)]\prod_{i\in I_x}\frac{(1-\BL^{-1})\BL^{u_iw(\FN_{\mcM/D_i,x})}}{1-\BL^{-1+u_i}} \]
in $\Ko(\Var_K)_{loc}$, 
where $B\Aut(x,\phi)$ denotes the classifying stack of the automorphism group of $(x,\phi) \in I_{\mu_N}\mcM$ and $\FN_{\mcM/D_i,x}$ denotes the fiber at $x$ of the normal bundle of $D_i$ in $\mcM$.
\end{theorem}

\begin{proof} Let $[X/\GL_d] \to \mcM$ be a Nisnevich cover and fix a lift $(\tilde{x},\tilde{\phi}) \in I_{\mu_N}[X/\GL_d]$ of $(x,\phi)$, which exists by the same argument as for \eqref{glsr}. We may then identify $e_N^{-1}(x,\phi)$ with $e_N^{-1}(\tilde{x},\tilde{\phi}) \subset [X/\GL_d](D_K^{1/N})$ and this identification is measure preserving by definition of the motivic measure. 

We thus assume from now on that $\mcM = [X/\GL_d]$ and let $q: X \to \mcM$ be the smooth atlas. Then we have $p^{-1}(e_N^{-1}(x,\phi)) =  X(K\llb{t^{1/N}}\rrb)_{p^{-1}(x)}^\phi$ and thus, by Proposition \ref{equivol},

\begin{align*} \int_{e_N^{-1}(x,\phi)}|\omega_D|  =  \frac{\BL^{w(\mathfrak{gl}_d)}}{[\Cent(\phi)]}&\int_{X(K\llb{t^{1/N}}\rrb)_{p^{-1}(x)}^\phi}|\omega_{\tilde D}|=\\ &\BL^{-w(x,\phi)}\frac{[p^{-1}(x)]}{[\Cent(\phi)]}\prod_{i\in I_x}\frac{(1-\BL^{-1})\BL^{u_iw(\FN_{\mcM/D_i,x})}}{1-\BL^{-1+u_i}}.\end{align*}

Finally we notice that $\Aut(x,\phi) = \Aut(x) \cap \Cent(\phi)$ and thus there is an equivalence of quotient stacks $[p^{-1}(x)/\Cent(\phi)] \xrightarrow{\sim} B\Aut(x,\phi)$. As $\Cent(\phi) \subset \GL_d$ is a product of general linear groups we have $[B\Aut(x,\phi)] = \frac{[p^{-1}(x)]}{[\Cent(\phi)]}$.
\end{proof}

We can globalize Theorem \ref{fixram} to any constructible subset $Z \subset I_{\mu_N}\mcM$. For any $J \subset I$ we write as usual $ D_J^\circ = \cap_{j\in J} D_j \setminus \cup_{i\notin J} D_i$ and furthermore 
\[I_{\mu_N} \mcM_J = \{(x,\phi) \in \mcM\ |\ x \in D_J^\circ \}.\]
We write $w_J: I_{\mu_N} \mcM_J \to \BQ^J$ for the constructible function given by $(w_J(x,\phi))_j = w(\FN_{\mcM/D_j,x})$ for all $j\in J$. 

Then for any constructible subset $Z \subset I_{\mu_N}\mcM$ we set

\[ [Z]^{w,D} = \sum_{J \subset I}\sum_{(a,b)\in\BQ \times \BQ^J} \BL^{-a}[w^{-1}(a)\cap w_J^{-1}(b)\cap Z] \prod_{j\in J}\frac{(1-\BL^{-1})\BL^{u_jb_j}}{1-\BL^{-1+u_j}} .  \]

\begin{corollary}\label{cor:orbi-form} For any constructible subset $Z \subset I_{\mu_N}\mcM$ we have
\[\int_{e_N^{-1}(Z)}|\omega_D| = [Z]^{w,D}.\]
\end{corollary}
\begin{proof}Both sides define motivic functions on $Z$ and by Theorem \ref{fixram} these functions agree on every point $\Spec(K) \to Z$. Around each point $\Spec(K) \to Z$ we can then find a constructible subset $U\subset Z$ where they agree. By compactness, finitely many such $U$ cover $Z$, showing that the two constructible functions are equal on $Z$.
\end{proof}

\subsection{Orbifold formula in an arbitrary theory}

In this section, we prove the orbifold formula in an arbitrary theory. This will not be used in the rest of the paper. 

Let $T$ be a theory in a language $\FC$ extending the Denef-Pas language $\FC_{DP,k}$ as in \cite[Section 2.7]{CL-2008} and containing the theory $T_{\mathrm{Hens}}$ of Henselian discretely valued fields of equicharacteristic zero. 

One obtains as in \cite[Section 16.1]{CL-2008} rings of constructible motivic functions that we denote in this section by $\FC^T(Z)$ for definable $Z$ in $T$. 

When $T$ is $T_{\mathrm{Hens},k}$, one obtains the ring denoted by $\FC(Z)$ in \cite{CL-2008}. When $T$ is the theory $T_\ACF$ obtained from $T_{\mathrm{Hens},k}$ by specifying that the residue field is algebraically closed, one obtains the rings considered in the rest of this paper. To avoid confusion, we denote them in this subsection by $\FC^\ACF(Z)$. When $T$ is the theory obtained from $T_{\mathrm{Hens}}$ by specifying that the residue field is pseudofinite, we denote the corresponding rings $\FC^\PSF(Z)$. 

If $i: T_1\to T_2$ is an inclusion of theories as above, and $Z$ a definable in $T_1$, then $Z$ determines a definable $i_*(Z)$ in $T_2$ since models of $T_1$ are models of $T_2$. As explained in \cite[Section 16.1]{CL-2008}, we then get a natural morphism $i_*\colon C^{T_1}(Z)\to C^{T_2}(i_*Z)$ compatible with integration by \cite[Proposition 16.1.1]{CL-2008}.

Now since $T_\ACF$ admits quantifier elimination in the Denef-Pas language, any definable $Z$ in $T_\ACF$ can be described by a quantifier free formula. This formula determines a definable $\iota_*Z$ in $T_{\mathrm{Hens},k}$, independent of the choice of the quantifier free formula representing $Z$. This assignment preserves bijections, hence induces a map $\iota_*\colon \FC^\ACF(Z)\to \FC(\iota_*Z)$. This morphism is compatible with integration in the same way as in \cite[Proposition 16.1.1]{CL-2008}, which follows by unravelling the construction of the integral. Moreover, the composite $\iota_*i_*$ is the identity for $i\colon T_{\mathrm{Hens},k}\to T_\ACF$. 

As in Section~\ref{sec:vol-twist-arcs}, let $\FM$ be a smooth cd-quotient stack over $k\llb{t}\rrb$. Let $D= \sum_{i\in I} u_i D_i$ a relative $\BQ$-Cartier divisor on $\mcM$ with $u_i< 1$ and simple normal crossing support, consider the motivic measure $|\omega_{\mcM,D}|=|\omega_D|$ on $\mcM(D^{1/N})$. 

Let $T$ be a theory as above, with inclusion $i\colon T_{\mathrm{Hens},k}\to T$. Then for any constructible subset $Z \subset I_{\mu_N}\mcM$, we consider $\iota_*i_*[Z]^{w,D}\in \FC^T(\iota_*i_*I_{\mu_N}\mcM).$

\begin{theorem}
\label{thm:orbiformulaT} For $T$ as above and $Z \subset I_{\mu_N}\mcM$ constructible, we have  
\[\int^T_{e_N^{-1}(Z)}|\omega_D| = [Z]^{w,D}\in \FC^T(\iota_*i_*I_{\mu_N}\mcM).\]
\end{theorem}
\begin{proof}
If $T$ is $T_\ACF$, this is Corollary \ref{cor:orbi-form}. 
In general, consider the morphism \[\iota_*i_*\colon \FC^\ACF(I_{\mu_N}\mcM)\longrightarrow  \FC^T(\iota_*i_*I_{\mu_N}\mcM)\] and the same for $\FM(D^{1/N})$.  They commute with integration by \cite[Proposition 16.1.1]{CL-2008} and the discussion above, so the result in general follows from the one for $T_\ACF$. 
\end{proof}

Recall from \cite[Definition 2.13.2]{FLW} the notion of definable smooth Deligne-Mumford stack over $k\llb{t}\rrb$, which applies to a theory $T$ as above. 

\begin{corollary}
\label{cor:orbiformulaT}
Let $\FM$ be a definable smooth Deligne-Mumford stack over $k\llb{t}\rrb$. Then for any constructible $Z\subset I_{\hat \mu}\mcM$, we have 
\[\int^T_{e^{-1}(Z)}|\omega_\mcM| = [Z]^{w}\in \FC^T(I_{\hat\mu}\mcM),\]
where 
\[ [Z]^{w}= \sum_{a \in\BQ} \BL^{-a}[w^{-1}(a)\cap Z]. \]
\end{corollary}
\begin{proof}
When $\FM$ is a smooth Deligne-Mumford stack over $k\llb{t}\rrb$, we have $I_{\hat \mu}\mcM=I_{\mu_N}\mcM$ for $N$ large enough such that all orders of stabiliser divide $N$, so this is a particular case of Theorem \ref{thm:orbiformulaT}. 

For any choice of model and parameters, points of a definable Deligne-Mumford stack are points of an actual Deligne-Mumford stack, with matching inertia stacks. So the general case reduce to this particular case we just proved.
\end{proof}

\begin{rmk}
This Corollary for $T_\PSF$ is the orbifold formula of \cite[Theorem 2.13.4]{FLW}.
\end{rmk}

\section{Hrushovski-Kazhdan characteristic in buildings }\label{hkc}
\subsection{Bruhat-Tits buildings}
In this section $F$ is a rank one valued field, with valuation $v$ and valuation ring $R$. 

\begin{definition}
An \emph{additive norm} $\alpha$ on an $F$-vector space $V$ 
is a map 
$\alpha \colon V\to \BR\cup\set{\infty}$ satisfying :
\begin{enumerate}
\item For every $v,w\in V$, $\alpha(v+w)\geq \min\set{\alpha(v),\alpha(w)}$.
\item For every $v\in V$, $\lambda\in F$, $\alpha(\lambda v)=v(\lambda)+\alpha(v)$.
\item For every $v\in V$, $\alpha(v)=\infty$ if and only if $v=0$.
\end{enumerate}
\end{definition}

For our purposes, we shall consider rational additive norms: assuming $F$ has value group included in $\BQ$, a \emph{rational additive norm} is an additive norm $\alpha$ such that the image of $\alpha$ is included in $\BQ\cup \set{\infty}$. More generally, assuming $F$ has value group included in some ordered group $\Gamma$, we can consider $\Gamma$-additive norms, that is, additive norms with values in $\Gamma$.

\begin{definition}
An additive norm $\alpha$ \emph{splits over $v_1,\dots,v_r\in V$} if for every $\lambda_1,\dots,\lambda_r\in F$, 
\[
\alpha\Bigl(\sum_i \lambda_i v_i\Bigr)=\min_i\set{\alpha(\lambda_i v_i)}.
\]
\end{definition}
A splitting set for $\alpha$ is necessarily linearly independent, and we call a splitting set which is also a basis of $V$ a splitting basis of $\alpha$. 

Define the ball $L_{\alpha,r}=\set{v\in V\mid \alpha(v)\geq r}$. It is an $R$-module. From \cite[Lemma 15.1.10]{KP:BTT}, $\alpha$ admits a splitting basis if and only if $L_{\alpha,r}$ is a lattice, that is, a finitely generated $R$-module that generates $V$ over $F$.
Moreover, if $F$ is complete then every additive norm $\alpha$ splits over some basis, by \cite[Proposition 15.1.11]{KP:BTT}.

\begin{definition}
\begin{enumerate}
\item Given a basis $v_1,\dots,v_n$ of $V$, the apartment associated to it is the set 
$\widetilde{\FA}(v_1,\dots,v_n)$ of rational additive norms that splits over $v_1,\dots,v_n$. 
\item The enlarged building $\widetilde{\FB}(\GL_n(F))$ of $\GL_n(F)$ is the set of rational additive norms on $F^n$ splitting  over some basis. It is a union of the apartments associated to  all bases of $F^n$. 
\end{enumerate}
\end{definition}

The building $\widetilde{\FB}(\GL_n(F))$ has a structure of a simplicial complex by considering periodic lattice chains, as follows. 

\begin{definition}
A periodic lattice chain $\mathfrak{L}$ is a set of lattices totally ordered by inclusion and such that if $\FL\in \mathfrak{L}$, then $x\cdot L\in \mathfrak{L}$ for every $x\in F^\times$. 

A grading of $\mathfrak{L}$  is a map $c\colon \mathfrak{L}\to \BR$ such that $c(\mathfrak{M}\FL)=c(\FL)+1$. 
\end{definition}

There is a one to one correspondence between the set of splitting additive norms and the set of graded periodic lattice chain. To $\alpha$, one associates the lattice chain $(L_{\alpha,r})_{r\in \BR}$, together with the map $c(L)=\inf\set{\alpha(v)\mid v\in L}$. Reciprocally, given $(\mathfrak{L},c)$, one defines $\alpha(v)=\sup\set{c(L)\mid v\in L\in \mathfrak{L}}$.

We say that two points of $\widetilde{\FB}(\GL_n(F))$ lie in the same simplex if they correspond to the same lattice chain. We say that the simplex corresponding to a chain $\mathfrak{L}$ is a face of the one corresponding to $\mathfrak{L}'$ if $\mathfrak{L}\subset \mathfrak{L}'$. In particular, the set of vertices of $\widetilde{\FB}(\GL_n(F))$ is identified with the set of $R$-lattices of $F^n$, which is $\GL_n(F)/\GL_n(R)$.

Fix a lattice chain $\mathfrak{L}$ and $L_0\in \mathfrak{L}$. By taking quotients modulo $\mathfrak{M}L_0$, the set of $L\in \mathfrak{L}$ such that $\mathfrak{M}L_0\subsetneq L\subsetneq L_0$ is mapped to the set of elements of a partial flag in the vector space $k^n$, where $k$ is the residue field of $F$. If such a set is of cardinality $r$, then we call $r+1$ the rank of $\mathfrak{L}$, and it is the rank of the simplex in which $\mathfrak{L}$ lies. Hence the simplicial structure is determined by flag varieties over the residue field. For example, the set of $n$-simplices adjacent to a given $n-1$ simplex is identified with $\BP^1(k)$; the set of $n$-simplices adjacent to a given vertex is identified with the complete flag variety over $k^n$.

The building is equipped with a metric $d$ defined using the affine structure on apartments, which are affine spaces over $\BR^n$.

We have a natural action of $\GL_n(F)$ on $\widetilde{\FB}(\GL_n(F))$, given by $(g\cdot \alpha)(v)=\alpha(gv)$.

Let $F'$ be a valued field extension of $F$. We also have an action of the Galois group $\Gal(F'/F)$ on $\widetilde{\FB}(\GL_n(F'))$, through the action of $\GL_n(F')$.

\begin{proposition}[{\cite[Theorem 12.9.2]{KP:BTT}}]
\label{prop:descent-BT}
There is a  natural identification between the sets  $\widetilde{\FB}(\GL_n(F))$ and $\widetilde{\FB}(\GL_n(F'))^{\mathrm{Gal}(F'/F)}$. 
\end{proposition}
Note that the identification in Proposition \ref{prop:descent-BT} does not preserve the simplicial structure, since some vertices are not sent to vertices.

From now on we restrict our attention to rational points of buildings, that is, rational additive norms. We write $\widetilde{\FB}^\BQ(\GL_n(F))$ (resp. $\widetilde{\FB}^\Gamma(\GL_n(F))$)  for the extended building consisting of rational additive norms in $\widetilde{\FB}(\GL_n(F))$ (resp. $\Gamma$-additive norms).

\begin{proposition}
\label{prop:descent-BT-Finf}
There is a  natural identification between the sets  $\widetilde{\FB}^\BQ(\GL_n(F))$ and $(\GL_n(F_\infty)/\GL_n(R_\infty))^{\mathrm{Gal}(F_\infty/F)}$. 
\end{proposition}
\begin{proof}
This follows either from a more general version of Proposition~\ref{prop:descent-BT} for buildings over non-discretely valued fields, or directly from Proposition~\ref{prop:descent-BT}, since 
\[(\GL_n(F_\infty)/\GL_n(F))^{\mathrm{Gal}(F_\infty/F)}=\bigcup_r (\GL_n(F_r)/\GL_n(F))^{\mathrm{Gal}(F_r/F)}\] and the observation that vertices are sent to rational additive norms by Proposition~\ref{prop:descent-BT}.
\end{proof}

Let $B_{<r}=\set{\alpha\in \widetilde{\FB}^\Gamma(\GL_n(F))\mid d(\alpha,\id)< r}$,
where $\id$ is the point in $\widetilde{\FB}^\Gamma(\GL_n(F))$ corresponding to the identity matrix.

\begin{lemma}
\label{lem:Br-def}
The set $B_{<r}$ is definable in $\Gamma\cup k$, and for each $i$, the set of $i$-simplices of $B_{<r}$ is definable. The same is true of $B_{\leq r}$.
\end{lemma}
\begin{proof}
We work by induction on $r$. 
Write $B_{<r}$ as a union over apartments containing $\id$. Each piece is definable since it is identified with the set of points in $\Gamma^n$ of norm less than $r$. Moreover, each piece contains only finitely many vertices, since those are identified with integer points in $\Gamma^n$. In particular, all simplices in $B_{<1}$ are adjacent to $\id$. Recall that for $i\geq 1$, the set of $i$-simplices adjacent to a given vertex is identified with a partial flag variety $\FP_i$ over $k^n$. Hence if $\Delta_i$ denotes the standard open $i$-simplex in $\Gamma^i$, and we set by convention $\FP_0=\set{1}$, $B_{<1}$ can be written as a disjoint union 
\[B_{<1}=\bigsqcup_{0\leq i \leq n}\FP_i\times\Delta_i,\]
hence is definable. 

Assume now $B_{<r}$ is definable, as well as its simplices. Then the set of $n-1$ simplices of $B_{<r+1}$ at distance $r$ of $\id$ is in bijection with the set of $n$-simplices of $B_{<r}$ at distance between $r-1$ and $r$ of $\id$, hence is definable by induction. Each such $n-1$ simplex, has $\BP^1(k)$ $n$ simplices adjacent, of which one of them is at distance between $r-1$ and $r$ of $\id$, the other $\BA^1(k)$ being at distance between $r$ and $r+1$. Hence the set of $n$ simplices at distance between $r$ and $r+1$ is definable. We see similarly that the set of $i$ simplices at distance between $r$ and $r+1$ is definable, hence $B_{<r+1}$ is definable. 
\end{proof}

\begin{proposition}
\label{prop:Volbuilding}
Let $X$ be a definable subset of $\GL_n(F_\infty)/\GL_n(R_\infty))^{\mathrm{Gal}(F_\infty/F)}$. Assume that $X$ is definably compact and definably connected. Then $\Vol(X)=1$. 
\end{proposition}
The intuition behind the proof is that such an $X$ should be contractible.
\begin{proof}
By Proposition \ref{prop:descent-BT-Finf}, one can view $X$ as a definable subset of $\widetilde{\FB}^\Gamma(\GL_n(F))$. 
Since $X$ is definably compact, it is bounded, hence is contained in $B_{\leq r}$ for some $r$. Up to translating $X$, we can assume that $\id\in X$. It follows from  Lemma~\ref{lem:Br-def} that $B_{\leq r}$ is definable. Assume $r$ is minimal among integers such that $X\subset B_{\leq r}$. If $r=0$, then $X=\set{\id}$ and we are done. Assume that $r\geq 1$. We shall show that $X\cap B_{\leq r-1}$ is closed, bounded and definably connected and $\Vol(X\cap B_{\leq r-1})=\Vol(X)$. Hence we can conclude inductively by applying the proposition to $X\cap B_{\leq r-1}$. 

The set $X \cap B_{\leq r-1}$ is clearly bounded and closed. It is also connected, since $X$ is connected and an intersection of apartments is connected.

So it remains to show that $\Vol(X\cap B_{\leq r-1})=\Vol(X)$. We will rather show that $\Vol(X\cap C_{r-1<,\leq r})=0$, where 
$C_{r-1<,\leq r}=\set{\alpha\in \widetilde{\FB}^\Gamma(\GL_n(F))\mid r-1< d(\alpha,\id)\leq r}$.

Recall the following  general properties of buildings.
If two chambers share two faces, they are equal. A chamber at distance between $r-1$ and $r$ of $\id$ has exactly one face at distance $r-1$ of $\id$. A face at distance $r-1$ of $\id$ has $\BA^1(k)$ adjacent chambers at distance between $r-1$ and $r$ and one adjacent chamber at distance between $r-2$ and $r-1$. 

From those properties, it follows that $X\cap C_{r-1<,\leq r}$ is a disjoint union
\[
X\cap C_{r-1<,\leq r}=\bigsqcup_{(z_0,z_1)\in Z\times \BA^1(k)} X\cap (\bar C_z\backslash \bar z_0),
\]
where $Z$ is the set of faces at distance $r-1$ of $\id$, and for $z=(z_0,z_1)$, $C_z$ is a chamber at distance between $r-1$ and $r$ which is adjacent to $z_0$, and $\bar C_z$ is its closure. The set $Z$ is definable by Lemma~\ref{lem:Br-def}. The set $X\cap \bar C_z$ is closed, bounded and connected inside $\Gamma^r$, hence if it is non-empty, $\Vol(X\cap \bar C_z)=1$. But if $X\cap \bar C_z\neq \emptyset$, then $X\cap \bar z_0\neq \emptyset$ by connectedness of $X$, hence $\Vol(X\cap \bar z_0)=1$ and $X\cap (\bar C_z\backslash \bar z_0)=0$. 
Hence $\Vol(X\cap C_{r-1<,\leq r})=0$. 
\end{proof}

\section{Generalized orbifold formula}\label{gof}

\subsection{The main result}
Let $\mcM$ be a smooth cd-quotient stack over $k\llb{t}\rrb$ and  $D= \sum_{i\in I} u_i D_i$ a relative $\BQ$-Cartier divisor on $\mcM$ with $u_i< 1$ and simple normal crossing support. 
Let
\[ \pi: \mcM \longrightarrow M\]
be a morphism to a finite type scheme $M$ over $k\llb{t}\rrb$. Throughout this section have the following additional
\begin{assumption}\label{basas}\begin{enumerate}
\item The morphism $\pi$ is $\FS$-complete. 
\item\label{gtors} There exists a dense open $U \subset M$ such that $\pi^{-1}(U) \to U$ is an equivalence. 
\item\label{propex} For any algebraically closed field $K /k$ and any $x\in M(K\llb{t}\rrb) \cap U(K\llp{t}\rrp)$ there exists an integer $N \geq 1$ and an $\tilde{x} \in \mcM(D^{1/N})$ such that the following diagram commutes:
\[\xymatrix{
D^{1/N} \ar[d] \ar[r]^{\ \ \ \tilde{x}} & \mcM \ar[d]^\pi\\
\Spec(K\llb{t}\rrb) \ar[r]^{\ \ \ \ x} & M.}\]
 
\item\label{ginvb} Some power of $\omega_\mcM(D)$ is Zariski-locally trivial, where $\omega_\mcM$ denotes the canonical bundle of $\mcM$ i.e. the determinant of its cotangent complex. 

\end{enumerate}
\end{assumption}

\begin{rmk}\label{erm} If $ \pi: \mcM \to M$ is a good moduli space morphism satisfying (2), then (1) and (3) are automatically satisfied. Indeed, good moduli space morphisms are $\FS$-complete \cite[Proposition 3.44]{AHH} and (3) follows since $\mcM$ admits a birational morphism $\mcM' \to \mcM$ from a Deligne-Mumford stack \cite[Theorem 2.11]{ER21}. 
\end{rmk}

Under these assumptions we may define a motivic measure $|\omega_{M^\natural,D}|$ on the definable $M^\natural$ given by the assignment
\[ K/k \longmapsto  M(K\llb{t}\rrb) \cap U(K\llp{t}\rrp).\]

For the construction of $|\omega_{M^\natural,D}|$ consider a positive integer $r$ and a Zariski cover $\mcM = \bigcup_i \FV_i$, such that $\omega_\mcM(D)^{\otimes r}_{|\FV_i}$ is trivial. If we choose for each $i$ a trivializing section $\omega_i$ of $\omega_\mcM(D)^{\otimes r}$ on $\FV_i$, then $\omega_i$ restricts to a pluricanonical form on $U \cap \FV_i$ and thus defines a measure $|\omega_i|$ on the assignment $M^\natural_i$ defined by
\[ K/k \longmapsto  \{x\in M(K\llb{t}\rrb) \cap U(K\llp{t}\rrp) \ |\ \exists N\geq1 \colon x \text{ lifts to an } \tilde{x} \in \FV_i(D_K^{1/N})\}.\]

Note that by the Assumption \ref{basas}(3) we have a covering $M^\natural =  \cup_i M_i^\natural$.


\begin{lemma} The $ |\omega_i|$ glue to a motivic measure $|\omega_{M^\natural,D}|$ on $M^\natural =  \cup_i M_i^\natural$ independent of the choices of trivializing sections $\omega_i$. 
\end{lemma}
\begin{proof} The argument is very similar to \cite[Proposition 3.1.2]{COW24}. Any two sections $\omega_i$ and $\omega_j$ differ on $\FV_i \cap \FV_j$ by an invertible function $f_{ij}$. We need to show that for every $K/k$ and every $x \in M_i^\natural \cap M_j^\natural$ we have $|f_{ij}(x)| =1$. This follows since $x$ lifts to an $\tilde{x} \in \FV_i\cap \FV_j(D^{1/N})$ for some $N\geq 1$ and $|f_{ij}(\tilde{x})| =1$ since $f_{ij}$ is invertible.
\end{proof}

We will express integrals with respect to $|\omega_{M^\natural,D}|$ on $M$ in terms of $I_{\mu_N}\mcM$ for all $N\geq 1$. Recall from Section~\ref{sec:vol-twist-arcs} that for any constructible subset $Z \subset I_{\mu_N}\mcM$ we define
\[ [Z]^{w,D} = \sum_{J \subset I}\sum_{(a,b)\in\BQ \times \BQ^J} \BL^{-a}[w^{-1}(a)\cap w_J^{-1}(b)\cap Z] \prod_{j\in J}\frac{(1-\BL^{-1})\BL^{u_jb_j}}{1-\BL^{-1+u_j}} .  \]

We also write $M^\natural_Z$ for the assignment
\[ K/k \longmapsto  \{ x \in M(K\llb{t}\rrb)^\natural \ |\ x_{|\Spec(K)} \in Z\}.\]


\begin{theorem}\label{mainlim}
Let $\mcM\to M$ be a morphism satisfying Assumption~\ref{basas}. Then for any constructible  subset $Z \subset M_k$, the series $\sum_{N\geq 1} [I_{\mu_N} \mcM_Z]^{w,D} \, T^N$ is rational of degree zero, hence its limit as $T\to+\infty$ exists formally and we have
\[
\int_{M^\natural_Z}|\omega_{M^\natural,D}|=-\lim_{T\to +\infty} \sum_{N\geq 1} [I_{\mu_N} \mcM_Z]^{w,D} \, T^N
\]
in $\Ko(\Var_k)_{loc}$, where $I_{\mu_N} \mcM_Z = \{(y,\phi) \in I_{\mu_N} \mcM\ |\ \pi(y) \in Z\}$.
\end{theorem}

For convenience let us also spell out Theorem \ref{mainlim} for $D=0$, in which case we write $|\omega_{M^\natural}|$ instead of $|\omega_{M^\natural,0}|$.
For any constructible subset $Z \subset I_{\mu_N}\mcM$ we set 
\[ [Z]^{w}= [Z]^{w,0} = \sum_{a \in\BQ} \BL^{-a}[w^{-1}(a)\cap Z]. \]

\begin{corollary}\label{noD}Let $\mcM\to M$ be a morphism satisfying Assumption~\ref{basas}. Then we have for any constructible  subset $Z \subset M_k$
\[
\int_{M^\natural_Z}|\omega_{M^\natural}|=-\lim_{T\to +\infty} \sum_{N\geq 1} [I_{\mu_N} \mcM_Z]^{w} \, T^N
\]
in $\Ko(\Var_k)_{loc}$.
\end{corollary}

Finally we record the following lemma for later.

\begin{lemma}
\label{lem:measfubini} Consider $\pi_N:\FM(D_K^{1/n})^\natural \to M(K\llb{t}\rrb)^\natural$.
 Both spaces come with natural measures $|\omega_{\mcM,D}|$ and $|\omega_{M^\natural,D}|$ respectively and we have for every $x \in M(K\llb{t}\rrb)^\natural$ and every measurable $A \subset \pi_N^{-1}(x)$ the equality
\[\int_{A} \frac{|\omega_{\mcM,D}|}{\pi_N^{*}|\omega_{M^\natural,D}|} = \mu^\#(A).\]
\end{lemma}
\begin{proof}
As before we may reduce to the case where $\FM \cong [X/\GL_d]$.  By definition of the measure on $\FM(D_K^{1/n})^\natural$ we need to consider the decomposition 
\[A=\bigsqcup_{[\phi] \in \Hom(\mu_N,\GL_d)/\GL_d}A^\phi, \]
where $A^\phi = A \cap \overline{e}^{-1}_N(\phi)$. If we write $q: X(K\llb{t^{1/N}}\rrb)^\phi  \to \overline{e}^{-1}_{N}(\phi)$ for the quotient then we have by definition 

\[\int_{A^\phi} \frac{|\omega_{\mcM,D}|}{\pi_N^{*}|\omega_{M^\natural,D}|} =\frac{\BL^{w(\mathfrak{gl}_d)}}{[\Cent(\phi)]}\int_{q^{-1}(A^\phi)} \frac{|\omega_{X,\tilde D}|}{\pi_N^{*}|\omega_{M^\natural,D}|}. \]

Now recall that $|\omega_{M^\natural,D}|$ is defined by integrating locally sections of some power of $\omega_\FM(D)$, while $|\omega_{X,\tilde D}|$ is given by locally integrating sections of some power of $\omega_X(\tilde D)$. Since $X \to \FM$ is a $\GL_d$ principal bundle the quotient $\frac{|\omega_{X,\tilde D}|}{\pi_N^{*}|\omega_{M^\natural,D}|} $ is thus given by locally integrating a nowhere vanishing $\GL_d$-invariant form $|\omega_d|$ on $q^{-1}(\pi_N^{-1}(x)^\phi)$. Consider the definable map
\[\tilde h= h \circ q \colon q^{-1}(\pi_N^{-1}(x)^\phi) \longrightarrow \mathcal{B}_{\GL_d}(K\llp{t}\rrp),\]
whose fibers are free $\GL_d(K\llb{t^{1/N}}\rrb)^\phi$ orbits. Since 
\[ \int_{\GL_d(K\llb{t^{1/N}}\rrb)^\phi} |\omega_{\GL_d}| = \BL^{-w(\mathfrak{gl}_d)}[\Cent(\phi)],\]
we thus deduce again by Fubini that 
\[\int_{q^{-1}(A^\phi)} \frac{|\omega_{X,\tilde D}|}{\pi_N^{*}|\omega_{M^\natural,D}|}  =  \BL^{-w(\mathfrak{gl}_d)}[\Cent(\phi)] \mu^\#(\tilde{h}(q^{-1}(A^\phi))) =  \BL^{-w(\mathfrak{gl}_d)}[\Cent(\phi)]  \mu^\#(h(A^\phi)).\]
Since the map $h$ is a definable bijection onto its image the lemma follows from summing over all $[\phi] \in \Hom_k(\mu_N,\GL_d)/\GL_d$.
\end{proof}


\subsection{Volume of fibers}
We study in this section how various volumes of the fibers of $\FM\to M$ behave in families, and establish the key Proposition~\ref{prop:vol-lim} relating the Hrushovski-Kazhdan volume of a fiber to the point counting of its rational points. We assume in this section that $\FM$ is a cd-quotient stack.  Since the arguments are local in the target $M$, we can assume throughout the proofs of this subsection that $\FM$ is a quotient stack $[X/\GL_d]$.




For $K$ algebraically closed, set $M(K\llb t \rrb)^\natural=M(K\llb t \rrb)\cap U(K\llp t \rrp)$, which is represented by a definable set $\M^\natural$. Then define $\FM(K\llb t^{1/N} \rrb)^\natural$ as the subset of $\FM(K\llb t^{1/N} \rrb)$ such that its image in $M(K\llb t \rrb)$ lands in $M(K\llb t^{1/N} \rrb)^\natural$. Define similarly $\FM(D_K^{1/N})^\natural$, $\FM(K\llb t^{1/\infty} \rrb)^\natural$, \emph{etc.} by requiring that their images by $\pi$ land in $M(K\llb t^{1/N} \rrb)^\natural$, resp. $M(K\llb t^{1/\infty} \rrb)^\natural$.

\begin{proposition}
\label{prop:FMsharpdef}
There is an $\ACVF_{k\llp t \rrp}$ imaginary set $\FM^\natural$ and a definable morphism $\pi\colon \FM^\natural\to M^\natural$ such that for every algebraically closed valued field $L$ with valuation ring $\FO_L$, $\FM^\natural(L)=\FM(\FO_L)^\natural$
and 
\[
\pi_L\colon \FM(\FO_L)^\natural \longrightarrow M(\FO_L)^\natural
\]
is the map induced by $\pi\colon \FM\to M$. 
\end{proposition}

\begin{proof}
The stack $\FM$ is a global quotient $\FM=[X/G]$, with algebraic maps $p\colon X\to M$ and $\pi\colon\FM\to M$. Let $W=p^{-1}(U)$. Then $\FM(\FO_L)^\natural=W(\FO_L)/G(\FO_L)$, so it has a structure of an imaginary set. 
The morphism $\pi_L$ is defined using the algebraic morphism $p$, so $\pi_L$ is also definable. 
\end{proof}

\begin{lemma}
\label{lem:FMsharptwisted}
Let $K$ be an algebraically closed field. The preimage of $M(K\llb t \rrb)^\natural$ under 
\[\pi\colon \FM^\natural(K\llb t^{1/\infty} \rrb)\longrightarrow M(K\llb t^{1/\infty} \rrb)^\natural
\]
is $\FM(D_K^{1/\infty})^\natural$, and the preimage of $M(K\llb t \rrb)^\natural$ under 
\[\pi_N \colon \FM^\natural(K\llb t^{1/N} \rrb)\longrightarrow M(K\llb t^{1/N} \rrb)^\natural
\]
is $\FM(D_K^{1/N})^\natural$.
\end{lemma}
\begin{proof}
By definition, a point in $\FM^\natural(D^{1/\infty}_K)$ is a $K\llb t \rrb$-morphism from $\Spec(K\llb t^{1/\infty} \rrb)$ to $\FM$ which is $\hat\mu$ invariant. So $\FM^\natural(D^{1/\infty}_K)$ is the subset of $\FM^\natural(K\llb t^{1/\infty} \rrb)$ invariant by $\hat \mu$.

An arc in $\FM^\natural(K\llb t^{1/\infty} \rrb)$ is mapped by $\pi$ to a point in  $M(K\llb t \rrb)^\natural$ precisely when it is invariant under the action of $\hat \mu$, hence it corresponds to a twisted arc.  

The same holds at fixed ramification, replacing $K\llb t^{1/\infty} \rrb$ by $K\llb t^{1/N} \rrb$ and $\hat \mu$ by $\mu_N$. 
\end{proof}


\begin{lemma}
\label{lem:fibbounded}
The fibers of $\pi\colon \FM^\natural\to M^\natural$ are definably compact. 
\end{lemma}
\begin{proof}
We first show the result for the fibers of $p\colon X^\natural \to M^\natural$. By the characterisation of \cite[Lemma 4.2.4]{HL16}, we need to show that they are bounded and closed for the valuation topology. They are bounded, since included in the bounded set $X(K\llb t^{1/\infty} \rrb)$. They are also closed since $p$ is continuous. The fibers of $\pi$ are quotient of the fibers of $p$ by the free action of $G(K\llb t^{1/\infty} \rrb)$, hence definably compact as well. 
\end{proof}

\begin{proposition}
\label{prop:vol-lim}
The maps $x\in M(K\llb t \rrb)^\natural\mapsto \Vol(\pi^{-1}(x))$ and $(x,N)\in M(\llb t \rrb)^\natural \times \BZ_{>0}\mapsto \mu^\#(\pi_N^{-1}(x))$  determine constructible motivic functions in $\mcC(M^\natural)$ and $\mcC(M^\natural \times \BZ_{>0})$. Moreover, for every $x\in M(K\llb t \rrb)^\natural$, 
\[
\Vol(\pi^{-1}(x))=-\lim_{T\to +\infty} \sum_{N\geq 1} \mu^\#(\pi_N^{-1}(x)) T^N.
\]
\end{proposition}
\begin{proof}
In view of Proposition~\ref{prop:FMsharpdef} and Lemma~\ref{lem:FMsharptwisted}, the result follows from Propositions~\ref{prop:defvolx}, \ref{prop:defphinx} and \ref{prop:vollimeq}. The measure on the fibers is the counting measure by Lemma~\ref{lem:measfubini}, and the fibers are bounded by Lemma~\ref{lem:fibbounded} hence proper invariant since by definition, they are invariant by $G(\FO)$.
\end{proof}

\begin{rmk}
The map $x\mapsto \mu^\#(\pi_1^{-1}(x))$ is related to the function $\mathrm{sep}_{\pi_1}$ constructed by Satriano and Usatine in \cite{SU}. More precisely, when $\pi_1^{-1}(x)$ is finite, $\mu^\#(\pi_1^{-1}(x))=\mathrm{sep}_{\pi_1}(x)$. See Example~\ref{ex:toroidal} for a situation where this holds. So Proposition~\ref{prop:vol-lim} shows that the function $\mathrm{sep}_{\pi_1}$ has measurable fibers.
When the fibers are infinite, $\mu^\#(\pi_1^{-1}(x))$ is a natural generalization of $\mathrm{sep}_{\pi_1}$.
\end{rmk}

\begin{lemma}
\label{lem:fibsection}
The map $\pi\colon \FM^\natural\to M^\natural$ admits a definable section. 
\end{lemma}
\begin{proof}
Let $x\in M(K\llb t^{1/\infty} \rrb)$. We need to find a definable point in the fiber $\pi^{-1}(x)$ whenever it is non-empty.

Recall that $\FM(K\llb t^{1/\infty}\rrb)^\natural=(X^\natural(K\llb t^{1/\infty}\rrb)/\GL_n(K\llb t^{1/\infty}\rrb))$.

Let $\tilde x\in \pi^{-1}(x)$. 
Let $T$ be the standard maximal torus of $\GL_d$. Consider the orbit $O_{\tilde x}$ of $\tilde x$ under $T(K\llp t^{1/\infty} \rrp)$:
\[
O_{\tilde x}=\set{t\tilde x\in \FM(K\llb t^{1/\infty}\rrb)^\natural  \mid t\in T(K\llp t^{1/\infty} \rrp)}.
\]
The orbit $O_{\tilde x}$ is in $\tilde x$-definable bijection with a subset $\Delta_{\tilde x}$ of $\Gamma^n$ such that 
\[
\Delta_{\tilde x}=\set{\bar t\in \Gamma^d=T(K\llp t^{1/\infty} \rrp)/T(K\llb t^{1/\infty} \rrb)\mid t\tilde x\in O_{\tilde x}}. 
\]
Since $\Delta_{\tilde x}$ is a non-empty definable set in an o-minimal structure, it admits a definable point $\gamma_{\tilde x}$. Changing the choice of $\tilde x$ by amounts to a translation of $\Delta_{\tilde x}$, so up to this translation $\Delta_{\tilde x}$  is a definable set $\Delta$, admitting a definable point $\gamma$ independent of $\tilde x$. The preimage of $\gamma$ in $O_{\tilde x}$ is a definable point in the fiber independent of $\tilde x$. 
\end{proof}

\begin{lemma}
\label{lem:fibbijSn}
There is some integer $d$, a definable subset $Z\subset M^\natural\times S_d$, and a definable bijection $h\colon \FM^\natural \to Z$ commuting with the maps to $M^\natural$.  
\end{lemma}
\begin{proof}
By the previous Lemma~\ref{lem:fibsection}, the map $\pi$ admits a definable section. So we can construct $h$ using this definable point $\tilde x$ in $\pi^{-1}(x)$. We then map 
\[g\tilde x\in \pi^{-1}(x)=\set{g\tilde x\in \FM(K\llb t^{1/\infty}\rrb)^\natural  \mid t\in \GL_d(K\llp t^{1/\infty} \rrp)}
\]
to the image of $g$ in $S_d=\GL_d(K\llp t^{1/\infty} \rrp)/\GL_d(K\llb t^{1/\infty} \rrb)$. 
\end{proof}

\begin{example}[Toroidal case]
\label{ex:toroidal}
Assume that $\FM=[X/T]$, with $T$ a dimension $d$ split torus. Then the definable $Z$ of Lemma \ref{lem:fibbijSn} is a subset of $M^\natural\times \Gamma^d$. For each $x\in M^\natural(K\llb t\rrb)$, $Z_x$ is a compact polyhedron of $\Gamma^d$, and, by definition, $\Vol(\pi^{-1}(x))=\eu(Z_x)$, the Euler characteristic of $Z_x$. Moreover, $\mu^\#(\pi_N^{-1}(x))$ is 
equal to $\#Z_x(\frac{1}{N}\BZ)$,
the number  of $\frac{1}{N}\BZ$-points of $Z_x$. So Proposition~\ref{prop:vol-lim} states that those numbers vary as constructible motivic functions, and that  
\[
\eu(Z_x)=-\lim_{T\to+\infty} \sum_{N\geq 1}\#Z_x\Bigl(\frac{1}{N}\BZ\Bigr)T^N.
\]
We will show in Proposition~\ref{prop:vol:1} that $\eu(Z_x)=1$, under our running Assumption~\ref{basas} on $\FM$. 
\end{example}

\subsection{Hrushovski-Kazhdan characteristic of a fiber}\label{hkfib}



\begin{proposition}
\label{prop:vol:1}
Let $x\in M(K\llb t \rrb)^\natural$ and $\pi^{-1}(x)$ be the preimage of $x$ in $\FM^\natural(K\llb t^{1/\infty} \rrb)$. Then $\Vol(\pi^{-1}(x))=1$. 
\end{proposition}

First we need the following

\begin{lemma}
\label{lem:vol:Gmy}
For every $x\in M(K\llb t \rrb)^\natural$, the fiber $\pi^{-1}(x)$ is definably connected.
\end{lemma}
\begin{proof}
Let $z,z'\in \pi^{-1}(x)$. We can assume that they lie in the image of $X(K\llb t^{1/n} \rrb)$ for some $n$. 
This defines a morphism $\overline{ST}_{R_n}\backslash\set{0}\to [X/G]$, where
\[
\overline{ST}_{K\llb t^{1/n} \rrb}=[\Spec( K\llb t^{1/n} \rrb[s,\tau]/(s\tau-t^{1/n})) /\BG_m].
\]
Since $[X/G]$ is $\FS$-complete, this morphism extends as a morphism $\overline{ST}_{K\llb t^{1/n} \rrb} \to [X/G]$. 

Moreover, $\overline{ST}_{K\llb t^{1/n} \rrb}(K\llb t^{1/\infty} \rrb)$ can be identified with a closed interval in the value group. The above morphism sends its endpoints to $z$ and $z'$, hence it defines a continuous map from $z$ to $z'$ with image in $\pi^{-1}(x)$. 
\end{proof}

\begin{proof}[Proof of Proposition \ref{prop:vol:1}] 
By the cd-quotient stack property, we can assume that $\FM=[X/\GL_d]$. By Lemma~\ref{lem:fibbounded}, $\pi^{-1}(x)$ is definably compact, and by the previous lemma it is definably connected. 

Since being definably compact is preserved under definable homeomorphism by \cite[Proposition 4.2.9]{HL16}, its image by the bijection $h$ of Lemma~\ref{lem:fibbijSn} is then a definably compact and connected subset of $S_d$. Since $\Vol$ is preserved by definable bijections, the result now follows from Proposition~\ref{prop:Volbuilding} applied to this image.
\end{proof}

\subsection{Proof of Theorem \ref{mainlim} }

From Proposition~\ref{prop:vol:1}, we have $\Vol(\pi^{-1}(x))=1$ for every $x\in  M(K\llb t \rrb)^\natural$, hence
\[
\int_{M_Z^\natural}|\omega_{M^\natural,D}|=\int_{x\in M_Z^\natural}\Vol(\pi^{-1}(x)|\omega_{M^\natural,D}|.
\]
By Proposition~\ref{prop:vol-lim}, 
\[
\int_{x\in M_Z^\natural}\Vol(\pi^{-1}(x)|\omega_{M^\natural,D}|=\int_{x\in M_Z^\natural}-\lim_{T\to +\infty} \sum_{N\geq 1} \mu^\#(\pi_N^{-1}(x))|\omega_{M^\natural,D}| T^N.
\]
By Proposition~\ref{prop:switch:lim:int}, we can switch the limit and integral to get
\[
\int_{x\in M_Z^\natural}-\lim_{T\to +\infty} \sum_{N\geq 1} \mu_N(\pi_N^{-1}(x))|\omega_{M^\natural,D}| T^N=-\lim_{T\to +\infty} \sum_{N\geq 1} \int_{x\in M_Z^\natural}\mu^\#(\pi_N^{-1}(x))|\omega_{M^\natural,D}|.
\]
Since $\pi_N\colon \mcM(D^{1/N})\to M^\natural$ has source the twisted arcs of order $N$, by the Fubini theorem \ref{thm:fubini} and Lemma \ref{lem:measfubini}, we have
\[
\int_{x\in M_Z^\natural}\mu^\#(\pi_N^{-1}(x))|\omega_{M^\natural,D}|=\int_{\mcM(D^{1/N})_Z^\natural}|\omega_{\mcM,D}|.
\]
By Theorem~\ref{fixram}, we then have 
\[
\int_{\mcM(D^{1/N})_Z}|\omega_{\mcM,D}|=[I_{\mu_N}\mcM_Z]^{w,D}.
\]

\subsection{Examples}\label{expls}

\subsubsection{Deligne-Mumford stacks}\label{dms} Let $\mcM$ be a smooth Deligne-Mumford stack over $k\llb{t}\rrb$ and $\pi:\mcM \to M$ the map to its coarse moduli space. We further assume that $\mcM$ is generically stabilizer-free and we take $D=0$. 

Then Assumption \ref{basas}(2) is satisfied and thus by Remark \ref{erm} also (1) and (3). Finally (4) follows since $\mcM$ admits an étale cover from a smooth scheme.  

From comparing the definitions one sees that $|\omega_{M^\natural}|= |\omega_{orb}|$  and Corollary \ref{noD} then implies for any constructible $Z \subset M$  the so-called orbifold formula \cite{DL2002, Ya06,LW19,FLW}

\[  \int_{M^\natural_Z}|\omega_{M^\natural}|= [I_{\hat\mu} \mcM_Z]^{w}.\]

To see this, first notice that for a DM stack, $I_{\hat{\mu}}\mcM = \colim_N \Hom(B\mu_N,\mcM)$ is of finite type since the colimit stabilizes. It follows that there is a finite stratification $I_{\hat{\mu}}\mcM_Z = \sqcup_i V_i$ such that for every $N \geq 1$ the stack $I_{\mu_N}\mcM_Z$ is a union of strata $V_i$. If we denote by $N_i$ the smallest integer $N$ such that $V_i$ is a stratum of $I_{\mu_N}\mcM_Z$ we get

\[ -\lim_{T\to +\infty} \sum_{N\geq 1} [I_{\mu_N} \mcM_Z]^{w} \, T^N =-\lim_{T\to +\infty} \sum_i  [V_i]^w \sum_{N\geq 1} T^{N_iN}= [I_{\hat\mu} \mcM_Z]^{w}.\]

\subsubsection{Non-trivial $D$} Consider $\mcM = [\BA^3/\BG_m]$ where $\BG_m$ acts on $\BA^3$ with weights $(1,1,-1)$. Let 
\[\pi: \mcM \longrightarrow \BA^2\]
be the good moduli space map induced by $(x,y,z) \mapsto (xz,yz)$. If we take $D=-\{z=0\}$ we can again check that Assumption \ref{basas}(2) by removing all the orbits in $\BA^3$ with $0$ in their closure. Then (1) and (3) follow from Remark \ref{erm} and $(4)$ since $zdx\wedge dy \wedge dz$ induces a global trivialization of $\omega_{\mcM}(D)$. 

Thus we can compute $ \int_{M^\natural}|\omega_{M^\natural,D}|$ using Theorem \ref{mainlim} as follows. We stratify $\mcM$ into the three pieces $V_1 =  \mcM\setminus D$, $V_2 = D \setminus B\BG_m$ and $V_3 = B\BG_m$. For $i=1,2$ the pieces are schematic, so their contribution is independent of $N$ given by
\[  [I_{\mu_N} V_i]^{w,D} = \begin{cases} 1 &\text{ if } i=1,\\ \BL^{-2} &\text{ if } i=2. \end{cases} \]
For $V_3$ we compute first $w(0,\phi)$ and $w(\FN_{\mcM/D,0})$ for $\phi \in \Hom(\mu_N,\BG_m) \subset (0,1]$:

\[w(0,\phi)= \begin{cases} 2 &\text{ if } \phi = 1,\\ 2\phi + (1-\phi) -1=\phi &\text{ if } \phi \neq 1,\end{cases} \ \ \ \ w(\FN_{\mcM/D,0})= \begin{cases} 1 &\text{ if } \phi = 1,\\ 1-\phi &\text{ if } \phi \neq 1.\end{cases}  \]
From this we deduce
\[  [I_{\mu_N} V_3]^{w,D} = \frac{\BL^{-2}}{\BL^2-1} + (N-1) \frac{1}{\BL^2-1}.   \]
Passing to the limit and adding up the three contributions we finally obtain 
\[\int_{M^\natural}|\omega_{M^\natural,D}|= 1 + \BL^{-2} - \BL^{-2} = 1.\]
Of course we could have computed this without Theorem \ref{mainlim} by checking that $|\omega_{M^\natural,D}|$ is the standard measure on $M = \BA^2$.

\section{Applications}\label{apps}

\subsection{Stringy $E$-function for moduli spaces of vector bundles}\label{stef}

To illustrate how one can apply Theorem \ref{mainlim} in practice we compute here the stringy $E$-function for the moduli space $M = M_{2,0}$ of semi-stable rank $2$, degree $0$ vector bundles on a smooth projective curve $C$ of genus $g\geq 3$, recovering Theorem 6.1 from \cite{Kiem_Li} whose original proof relied on computations of explicit resolutions. 

\subsubsection{Stringy $E$-functions}

Let $X$ be a normal $\BQ$-Gorenstein variety with log-terminal singularities over an algebraically closed field $k$. For any canonical divisor $K_X$ and $r\geq 1$ such that $rK_X$ is Cartier, and thus the associated invertible sheaf defines a measure $|\omega_{can}|$ on $X$ as in Section \ref{fam}. The construction first appeared in \cite[Section 3]{DL2002} where it was called the Gorenstein measure.

 Using a log-resolution, one can give an explicit formula for the total volume $\int_{X(k\llb{t}\rrb)} |\omega_{can}|$ and one sees in particular that $\int_{X(k\llb{t}\rrb)} |\omega_{can}|$ converges in $\Ko(\Var_K)_{loc}[\BL^{1/r}]$ if and only if $X$ has log-terminal singularities, see also \cite[Proposition 6.6.2]{Ya?}.  

In this case, assuming $k =\BC$,  we can consider the image of $\BL^{\dim X}\int_{X(k\llb{t}\rrb)} |\omega_{can}|$ under the Hodge-Deligne realization
\[ E:\Ko(\Var_\BC)_{loc}[\BL^{1/r}] \longrightarrow \BZ[u,v][(uv)^{-1/r}, (uv-1)^{-n}; \ n \geq 1],   \]
which sends the class of any algebraic variety $Y$ to $E(Y; u,v) = \sum_{p,q,i} (-1)^i h_c^{i;p,q}(Y)u^pv^q$, where the $h_c^{i;p,q}(Y)$ are the compactly supported mixed Hodge numbers of $Y$.
Again using a log-resolution one sees that this image agrees with the stringy $E$-function $E_{st}(X;u,v)$ of a log-terminal $X$ as defined by Batyrev \cite{batyrev_stringy}.

We will use the following criterion to compare $|\omega_{can}|$ with $|\omega_{M^\natural}|$. 

\begin{lemma}\label{codim2}
Assume $\pi: \cM \to M$ satisfies Assumption \ref{basas} with $D=0$ and furthermore that $M$ is $\BQ$-Gorenstein. If there is an open $U \subset M$ such that $\codim_M M \setminus U \geq 2$ and $\pi^{-1}(U) \to U$ is an equivalence, then $ |\omega_{can}| = |\omega_{M^\natural}|$. 
\end{lemma}
\begin{proof} 
This follows directly from the construction, since the local non-vanishing forms used to define $|\omega_{M^\natural}|$ extend by the codimension $2$ assumption to non-vanishing sections of some power of the canonical bundle on $M$, which are used to define $|\omega_{can}|$.
\end{proof}

\subsubsection{Vector bundles on curves}

Let $C$ be a smooth projective curve $C$ of genus $g \geq 2$. For any $r \geq 1$ and $d \in \BZ$ there is a smooth algebraic stack $\BM_{r,d}$ parametrizing semi-stable rank $r$ and degree $d$ vector bundles on $C$. Since the automorphism group of every vector bundle contains a central copy of $\BG_m$ coming from scalar automorphisms, we may consider the $\BG_m$-rigidification $\cM_{r,d}$ of $\BM_{r,d}$ in the sense of \cite{ACV03}. Finally we denote by $M_{r,d}$ the coarse moduli space, which is Gorenstein by \cite{DN89}.

\begin{proposition}\label{vbf} The good moduli space map $\cM_{r,d} \to M_{r,d}$ satisfies Assumption \ref{basas} for $D=0$. In particular 
\[\int_{ M_{r,d}^\natural} |\omega_{M_{r,d}^\natural}| =  -\lim_{T\to +\infty} \sum_{N\geq 1} [I_{\mu_N} \mcM_{r,d}]^{w} \, T^N. \]

Unless $g=2$ and $(r,d) = (2,d)$ with $d$ even, we have  $ |\omega_{can}| = |\omega_{M_{r,d}^\natural}|$.

\end{proposition}
\begin{proof} Since $g \geq 2$, the moduli space of stable bundles $M^{st} \subset M$ provides a dense open as in \ref{basas} (2). Conditions (1) and (3) follow from Remark \ref{erm}. In order to verify (4) for $D=0$ we use Luna's slice theorem for smooth stacks \cite[Theorem 1.2]{AHR20}. Namely every polystable point $x \in \cM_{r,d}$ admits an open neighborhood $U_x$ which, up to an étale covering, is isomorphic to an open neighborhood of $0$ in $[\Ext^1(x,x)/\Aut(x)]$. Since the standard form on the vector space $\Ext^1(x,x)$ is invariant under the $\Aut(x)$-action, it induces a non-vanishing trivialization of $\omega_{\cM_{r,d}|U_x}$.

Finally unless $g=2$ and $(r,d) = (2,d)$ with $d$ even, the complement of $M^{st}$ in $M$ has codimension at least $2$ which can be seen from a dimension calculation as for example in \cite[Lemma 4.3]{NR69} and thus  Lemma \ref{codim2} implies $ |\omega_{can}| = |\omega_{M_{r,d}^\natural}|$.
\end{proof}

In principle, arguments similar to \cite[Section 5.5]{GWZ24} allow us to compute the right hand side of Proposition \ref{vbf} for all pairs $r,d$ in terms of motivic classes of $\mcM_{r',d'}$ with $r' \leq r$ and $d/r = d'/r'$. However already the leading term $[\mcM_{r,d}]$ is quite tricky to determine as one has to resolve the Harder-Narasimhan recursion \cite{Za96} and it is unclear to the authors whether the formula would be meaningful in any way. Instead we illustrate the computation in the simplest non-trivial case.

\begin{corollary} For $g \geq 3$ we have
\[ \int_{ M_{2,0}(k\llb{t}\rrb)} |\omega_{can}| = \BL^{-4g+3} [\mcM_{2, 0}] -\BL^{-3g+2}(\BL-1)\left( [\mathcal{P}ic^0(C)]^2 - [\Sym_2(\mathcal{P}ic^0(C))]\right),\]
where $\mathcal{P}ic^0(C)$ denotes the Picard stack of $C$. In particular
\begin{align*}
E_{st}(M_{2,0};u,v) &=  
 \frac{(uv)^{g-1}(1-u)^{2g}(1-v)^g -(1-u^2v)^{2g}(1-uv^2)^g(1-u)^g(1-v)^g}{(1-uv)^2(1-(uv)^2)}  - \\ &\frac{(uv)^{g-1}}{2}\left(\frac{(1-u)^{2g}(1-v)^{2g}}{(1-uv)^2} +\frac{(1-u^2)^g(1-v^2)^g}{1-(uv)^2}  \right). 
\end{align*}

\begin{rmk}Note that the above formula for $E_{st}(M_{2,0};u,v)$ is used in \cite{SU2026}
 to exhibit a counterexample to Batyrev's conjecture on the non-negativity of stringy Hodge numbers.
\end{rmk}

\end{corollary}

\begin{proof} The proof goes along the lines of \cite[Section 5.5]{GWZ24}. A point of  $I_{\mu_N} \mcM_{2,0}$ corresponds to a pair of a rank $2$, degree $0$ vector bundle $E$ on $C$ together with a morphism $\phi: \mu_N \to \underline{\Aut}(E)$, where $\underline{\Aut}(E) = \Aut(E)/\BG_m$. Since we work over an algebraically closed field, there exists a lift $\tilde{\phi}: \mu_N \to \Aut(E)$ and thus $E$ decomposes into eigenspaces $E = \bigoplus_{\chi \in\BZ/N\BZ} E_{\chi}$ for this $\mu_N$-action, where we use $\mathrm{Irr}(\mu_N) = \BZ/N\BZ$. Choosing a different lift of $\phi$ results in shifting the indexing of the $E_\chi$ by a constant in $\BZ/N\BZ$. There are 3 cases to consider:

First if $E = E_\chi$ for a single $\chi \in \BZ/N\BZ$, which corresponds to $\phi$ being the trivial homomorphism. This gives a component of $I_{\mu_N}\mcM_{r,d}$ which can be identified with $\mcM_{2,0}$. The weight of the trivial homomorphism is simply the dimension of $\mcM_{2,0}$ i.e. $4g-3$. 

If $E = E_{\chi_1} \oplus E_{\chi_2}$ with $(\chi_1,\chi_2) \neq (\chi_2+a,\chi_1+a)$ for all $a \in \BZ/N\BZ$. In this case both $E_{\chi_1}$ and $E_{\chi_2}$ are degree $0$ line bundles and the only automorphisms of the pair $(E,\phi)$ are the scalar automorphisms on each factor. Thus for a fixed pair $(\chi_1,\chi_2)$ the corresponding component of  $I_{\mu_N} \mcM_{2,0}$ is isomorphic to the $\BG_m$-rigidification of $\mathcal{P}ic^0(C)\times \mathcal{P}ic^0(C)$, where $\mathcal{P}ic^0(C)$ denotes the Picard stack of $C$. Furthermore one can check that up to simultaneous scaling by $a \in \BZ/N\BZ$ there are $\lfloor \frac{N-1}{2}\rfloor$ such pairs $(\chi_1,\chi_2)$. The weight in this case can be computed as in \cite[Lemma 5.19]{GWZ24} to be $3g-2$.

Finally, if $E = E_{\chi_1} \oplus E_{\chi_2}$ with $(\chi_1,\chi_2) = (\chi_2+a,\chi_1+a)$ for some $a \in \BZ/N\BZ$, we must have that $N$ is even and $a = N/2$. In this case $E_{\chi_1}$ and $E_{\chi_2}$ are still degree $0$ line bundles, but switching the two factors is now an automorphism of the pair $(E,\phi)$. In other words the corresponding component of  $I_{\mu_N} \mcM_{2,0}$ is isomorphic to the $\BG_m$-rigidification of $\Sym_2\mathcal{P}ic^0(C)$. There is only one such component and the weight equals $3g-2$ as before. 

Putting everything together we get

\begin{multline*} \int_{M_{2,0}(k\llb{t}\rrb)} |\omega_{can}|  = \\ 
-\lim_{T\to +\infty} \sum_{N\geq 1} \BL^{-4g+3} [\mcM_{2,0}] + 
\lfloor \frac{N-1}{2}\rfloor\BL^{-3g+2}(\BL-1)[\mathcal{P}ic^0(C)]^2 + \BL^{-3g+2}(\BL-1)[\Sym_2(\mathcal{P}ic^0(C))] T^N \\
= \BL^{-4g+3} [\mcM_{2,0}] - \BL^{-3g+2}(\BL-1)[\mathcal{P}ic^0(C)]^2 + \BL^{-3g+2}(\BL-1)[\Sym_2(\mathcal{P}ic^0(C))].
\end{multline*}

The $E$-polynomial realization follows from the following formulas:

\[E(\mcM_{2,0};u,v)=  \frac{(1-u^2v)^g(1-uv^2)^g(1-u)^{g}(1-v)^{g} -(uv)^{g-1}(1-u)^{2g}(1-v)^{2g}}{(1-uv)(1-(uv)^2)},  \]
which follows from \cite[Theorem 5.16]{He10} and \cite[Proposition 4.4]{Te98}, see also \cite[Remark 4.2]{BD07}. Furthermore
\[E(\mathcal{P}ic^0(C);u,v) = \frac{(1-u)^g(1-v)^g}{uv-1},\]
and for any variety $X$,
\[E(\Sym_2(X);u,v)= \frac{1}{2} E(X;u,v)^2 + \frac{1}{2} E(X;u^2,v^2),\]
which can be deduced for example from \cite{Ma62}. 
\end{proof}

The formula for $E_{st}(M_{2,0};u,v)$ agrees up to a factor of $E(\mathcal{P}ic^0(C);u,v)=\frac{(1-u)^g(1-v)^g}{uv-1}$ with the formula in \cite[Theorem 6.1]{Kiem_Li}, which comes from the fact that in \textit{loc.cit} the moduli space of vector bundles with fixed determinant is considered.

\subsection{Quasi-torsors and klt-singularities}\label{tqt}

Let $(Y,\Delta)$ be a klt-pair. Let $G$ be a linear algebraic group acting on a normal variety  $X$ admitting a $G$-invariant morphism $\phi:X \to Y$ such that there exist open subsets $U \subset Y$ and $V=\phi^{-1}(U) \subset X$, whose complements have codimension at least two and such that 
\[ \phi_{|V}\colon V \longrightarrow U  \]
is a $G$-torsor. 

These conditions are for example satisfied if $\phi$ is a $G$-quasi-torsor in the sense of \cite{BM}. Because of the codimension $2$ assumption $(X,\phi^{-1}\Delta)$ is a still a log pair and we have the following

\begin{theorem}\label{bmthm} If $G=T$ is a torus, then $(X,\phi^{-1}\Delta)$ is a klt-pair. 
\end{theorem}

For $T$-quasi-torsors a similar theorem was proven in \cite{BM}. 

\begin{proof}
The proof uses a characterization of $(Y,\Delta)$ being klt in terms of motivic integration. To do so one considers the motivic measure $|\omega_\Delta|= |\omega_{\Oc(r(K_Y+\Delta))}|$ constructed from the invertible sheaf $\Oc(r(K_Y+\Delta))$ as in Section \ref{fam}, where $r\geq 1$ is such that $r(K_Y+\Delta)$ is Cartier.
Using a log-resolution one sees that $\int_{Y(k\llb{t}\rrb)} |\omega_{\Delta}|$ converges in $\Ko(\Var_K)_{loc}[\BL^{1/r}]$ if and only if $(Y,\Delta)$ is a klt-pair, see \cite[Proposition 6.6.2]{Ya?}.  

The statement is thus equivalent to showing that $\int_{X(k\llb{t}\rrb)} |\omega_{\phi^{-1}\Delta}|$ converges. By Fubini we can write
\[\int_{X(k\llb{t}\rrb)} |\omega_{\phi^{-1}\Delta}| = \int_{Y(k\llb{t}\rrb)} I_\phi |\omega_{\Delta}|, \]
where $I_\phi$ is the function on $Y(k\llb{t}\rrb)$ given by
\[y \mapsto \int_{\phi^{-1}(y)} |\omega_{\phi^{-1}\Delta}|/\phi^*|\omega_{\Delta}|.  \]
But for every $y \in Y(k\llb{t}\rrb) \cap U(k\llp{t}\rrp)$, the fiber $\phi^{-1}$ surjects onto the integer points $R_y$ of a bounded region inside the building of $T$ with fiber $T(k\llb{t}\rrb)$. In particular $R_y$ is a finite set. This implies $I_\phi(y) = [T] |R_y|$. 

Now the map from $y\mapsto |R_y|$ is a constructible motivic function. To see this, consider the map $[X/T]\to Y$. Then $|R_y|$ is the cardinal of the fiber above $y$ of this map. The arguments from Lemmas \ref{lem:fibsection} and \ref{lem:fibbijSn} apply to our situation, giving a definable bijection from $[X/T](K\llb{t}\rrb)$ to a definable subset of $Y\times \BZ^n$ commuting with maps to $Y$. So $|R_y|$ is the cardinal of a finite definable subset of the value group, so clearly a constructible motivic function. 

Now there is a definable function $f$ from $Y$ to $\BZ^r$ and a constructible motivic function $\varphi$ on $\BZ^r$ such that $|R_y|=\varphi(f(y))$. 

Using integration along $f$, there is a constructible motivic function $\psi\in \FC(\BZ^r)$ such that 

\[
\int_{Y(k\llb{t}\rrb)} |\omega_{\Delta}|=\int_{\BZ^r} \psi
\]
and
\[
\int_{Y(k\llb{t}\rrb)} I_\varphi |\omega_{\Delta}|=[T] \int_{\BZ^r} \varphi \psi.
\]

By hypothesis, $\int_{\BZ^r} \psi$ converges. We can write $\psi$ as a linear combination of integrable functions in $\FP(\BZ^r)$, hence we may assume that it is in $\FP(\BZ^r)$ itself. By construction, $\deg_\eL(\varphi(x))\leq 0$ for every $x\in \BZ^r$. Hence by Lemma~\ref{lem:product-integrable} below, $\int_{\BZ^r} \varphi\psi$ converges. Hence $\int_{X(k\llb{t}\rrb)} |\omega_{\phi^{-1}\Delta}|$ converges, which is what we had to prove. 
\end{proof}

Consider the function  $\deg_\eL: \BA\to \BZ\cup\set{-\infty}$, the unique multiplicative extension of the degree in $\eL$ map on  $\BZ[\eL]$. 
\begin{lemma}\label{lem:product-integrable}
Let $\psi\in \FP(\BZ^r)$ be a Presburger constructible function that is integrable. Let $\varphi\in \FP(\BZ^r)$ such that $\deg_\eL(\varphi(x))\leq N$ for some fixed $N$.

Then $\varphi\psi$ is integrable. 
\end{lemma}

\begin{proof}
From \cite[Proposition 4.5.6]{CL-2008}, $\psi$ is integrable if and only if $\lim_{\abs{x}\to +\infty}{\deg_\eL}(\psi(x))= -\infty$. 

We have
\[
\deg_\eL(\varphi(x)\psi(x))=\deg_\eL(\varphi(x))+\deg_\eL(\psi(x))\leq \deg_\eL(\psi(x))+N.
\]
So that $\lim_{\abs{x}\to +\infty}{\deg_\eL}(\varphi(x)\psi(x))= -\infty$, hence $\varphi \psi$ is integrable by \cite[Proposition 4.5.6]{CL-2008}. 
\end{proof}

It is a natural question, whether Theorem \ref{bmthm} extends to more general reductive groups $G$ \cite[Question 7.8]{BM}. From our point of view this amounts to understanding better  the set of integer points $R_y$ inside the building of $G$ as $y$ varies, which seems to be a difficult but interesting question.

\bibliographystyle{abbrv}
\bibliography{artin_stacks_bibli}
\end{document}